\documentclass[11pt]{article}

\usepackage{geometry}
\usepackage{jhtitle,amsbsy,amsmath,amsthm,amssymb,times,graphicx}
\usepackage{probmacros}
\usepackage{natbib}
\usepackage[justification=centering]{caption}
\usepackage{subcaption}

\def\baro{\vskip  .2truecm\hfill \hrule height.5pt \vskip  .2truecm}
\def\barba{\vskip -.1truecm\hfill \hrule height.5pt \vskip .4truecm}

\newcommand{\pcite}[1]{\citeauthor{#1}'s \citeyearpar{#1}}

\newtheorem{theorem}{Theorem}

\newtheorem{corollary}[theorem]{Corollary}

\newtheorem{proposition}[theorem]{Proposition}

\author{Qian Qin, James P. Hobert and Kshitij Khare \\ Department of Statistics \\ University of
	Florida}
\date{April 2019}
\begin{document}

\title{Estimating the spectral gap of a trace-class Markov operator}
		
	\maketitle
	
	\begin{abstract}
          The utility of a Markov chain Monte Carlo algorithm is, in
          large part, determined by the size of the spectral gap of
          the corresponding Markov operator.  However, calculating
          (and even approximating) the spectral gaps of practical
          Monte Carlo Markov chains in statistics has proven to be an
          extremely difficult and often insurmountable task,
          especially when these chains move on continuous state
          spaces.  In this paper, a method for accurate estimation of
          the spectral gap is developed for general state space Markov
          chains whose operators are non-negative and trace-class.
          The method is based on the fact that the second largest
          eigenvalue (and hence the spectral gap) of such operators
          can be bounded above and below by simple functions of the
          power sums of the eigenvalues.  These power sums often have
          nice integral representations.  A classical Monte Carlo
          method is proposed to estimate these integrals, and a simple
          sufficient condition for finite variance is provided.  This
          leads to asymptotically valid confidence intervals for the
          second largest eigenvalue (and the spectral gap) of the
          Markov operator.  
          In contrast with previously existing techniques, our method is not
          based on a near-stationary version of the Markov chain, which,
          paradoxically, cannot be obtained in a principled manner without bounds on the spectral gap.
          On the other hand, it can be quite expensive from a computational
          standpoint.  The efficiency of the method is studied both theoretically
          and empirically.
	\end{abstract}
	
	\section{Introduction}
	
	Markov chain Monte Carlo (MCMC) is widely used to estimate
        intractable integrals that represent expectations with respect
        to complicated probability distributions.  Let $\pi: S \to
        [0,\infty)$ be a probability density function (pdf) with
        respect to a $\sigma$-finite measure $\mu$,
        where~$(S,\mathcal{U},\mu)$ is some measure space.  Suppose we
        want to approximate the integral
	\[
	J := \int_S f(u) \pi(u) \mu(du)
	\]
	for some function $f: S \to \mathbb{R}$.  Then~$J$ can be
        estimated by $\hat{J}_m := m^{-1} \sum_{k=0}^{m-1} f(\Phi_k),$
        where $\{\Phi_k\}_{k=0}^{m-1}$ are the first~$m$ elements of a
        well-behaved Markov chain with stationary density $\pi(\cdot)$.
        Unlike classical Monte Carlo estimators, $\hat{J}_m$ is not
        based on iid random elements.  Indeed, the elements of the
        chain are typically neither identically distributed nor
        independent.  Given $\mathrm{var}_{\pi}f,$ the variance of
        $f(\cdot)$ under the stationary distribution, the accuracy of
        $\hat{J}_m$ is primarily determined by two factors: (i) the
        convergence rate of the Markov chain, and (ii) the correlation
        between the $f(\Phi_k)$s when the chain is stationary. These
        two factors are related, and can be analyzed jointly under an
        operator theoretic framework.
	
        The starting point of the operator theoretic approach is the
        Hilbert space of functions that are square integrable with
        respect to the target pdf, $\pi(\cdot)$.  The Markov
        transition function that gives rise to $\Phi =
        \{\Phi_k\}_{k=0}^{\infty}$ defines a linear (Markov) operator
        on this Hilbert space.  (Formal definitions are given in
        Section~\ref{MarkovOperators}.)  If $\Phi$ is reversible, then
        it is geometrically ergodic if and only if the corresponding
        Markov operator admits a positive \textit{spectral gap}
        \citep{roberts1997geometric,kontoyiannis2012geometric}.  The
        gap, which is a real number in $(0,1]$, plays a fundamental role
        in determining the mixing properties of the Markov chain, with
        larger values corresponding to better performance.  For
        instance, suppose~$\Phi_0$ has pdf $\pi_0(\cdot)$ such that
        $d\pi_0/d\pi$ is in the Hilbert space, and let~$d(\Phi_k;
        \pi)$ denote the total variation distance between the
        distribution of $\Phi_k$ and the chain's stationary
        distribution.  Then, if~$\delta$ denotes the spectral gap, we
        have
	\[
	d(\Phi_k;\pi) \leq C(1-\delta)^k
	\]
	for all positive integers~$k,$ where~$C$ depends on $\pi_0$
        but not on~$k$ \citep{roberts1997geometric}.  Furthermore,
        $(1-\delta)^k$ gives the maximal absolute correlation between
        $\Phi_j$ and $\Phi_{j+k}$ as $j \to \infty$.  It follows
        \citep[see e.g.][]{mira1999ordering} that the asymptotic
        variance of $\sqrt{m}(\hat{J}_m - J)$ as $m \to \infty$ is
        bounded above by
	\[
	\frac{2-\delta}{\delta} \mathrm{var}_{\pi}f \,.
	\]
	Unfortunately, it is impossible to calculate the spectral gaps
        of the Markov operators associated with practically relevant
        MCMC algorithms, and even accurately approximating these
        quantities has proven extremely difficult.  In this paper, we
        develop a method of estimating the spectral gaps of Markov
        operators corresponding to a certain class of data
        augmentation (DA) algorithms \citep{tanner1987calculation},
        and then show that the method can be extended to handle a much
        larger class of reversible MCMC algorithms.

        DA Markov operators are necessarily non-negative.  Moreover,
        any non-negative Markov operator that is compact has a pure
        eigenvalue spectrum that is contained in the set $[0,1]$, and
        $1-\delta$ is precisely the second largest eigenvalue.  We
        propose a classical Monte Carlo estimator of $1-\delta$ for DA
        Markov operators that are \textit{trace-class}, i.e. compact
        with summable eigenvalues.  While compact operators were once
        thought to be rare in MCMC problems with uncountable state
        spaces \citep{chan1994discussion}, a string of recent results
        suggests that trace-class DA Markov operators are not at all
        rare \citep[see
        e.g.][]{qin2018trace,chakraborty2016convergence,choi2017anova,pal2017trace}.
        Furthermore, by exploiting a simple trick, we are able to
        broaden the applicability of our method well beyond DA
        algorithms.  Indeed, if a reversible Monte Carlo Markov chain
        has a Markov transition density (Mtd), and the corresponding
        Markov operator is Hilbert-Schmidt, then our method can be
        utilized to estimate its spectral gap.  This is because the
        square of such a Markov operator can be represented as a
        trace-class DA Markov operator.  A detailed explanation is
        provided in Section~\ref{daint}.

	Of course, there is a large literature devoted to developing
        theoretical bounds on the second largest eigenvalue of a
        Markov operator \citep[see e.g.][]{lawler1988bounds,
          sinclair1989approximate, diaconis1991geometric}. However,
        these results are typically not useful in situations where the
        state space,~$S$, is uncountable or multi-dimensional, which
        is our main focus. 
        There also exist a number of computational methods for
        approximating the eigenvalues of a Hilbert-Schmidt operator
        \citep[see e.g.][]{garren2000estimating,koltchinskii2000random,ahues2001spectral,chakraborty2017consistent}.
        Some such methods require sampling directly from $\pi(\cdot)$, which is
        impossible in an MCMC context.  
        The others require the user to simulate
        the Markov chain of interest until it is nearly stationary.
        Unfortunately, we cannot know if a chain has converged unless we have
        some information on its convergence rate, which is essentially what
        these methods are trying to acquire in the first place.  
        The classical
        Monte Carlo estimator that we introduce is calculated by simulating many
        copies of the Markov chain, each of a short length. These short chains need not be close to stationarity in order for the
        estimator to be valid.  
        Although powerful, this method is quite expensive from a
        computational standpoint. 
        Indeed, it works well only when the
        underlying dataset of the Bayesian model is small.  
        On the other hand,
        it is important as a ``proof of concept" that it is actually
        possible to get a handle on the spectral gaps of Markov operators
        corresponding to MCMC algorithms on continuous state spaces, which,
        until now, have proven to be extremely elusive quantities.

	The rest of the paper is organized as follows.  The notion of
        Markov operator is formalized in
        Section~\ref{MarkovOperators}.  In Section~\ref{powersum}, it
        is shown that the second largest eigenvalue of a non-negative
        trace-class operator can be bounded above and below by
        functions of the power sums of the operator's eigenvalues.  In
        Section~\ref{daint}, DA Markov operators are formally defined,
        and the sum of the $k$th ($k\in \mathbb{N}$) power of the
        eigenvalues of a trace-class DA Markov operator is related to
        a functional of its Mtd.  This functional is usually a
        multi-dimensional integral, and a classical Monte Carlo
        estimator of it is developed in Section~\ref{MC}. The efficiency of the Monte Carlo estimator is studied in Section~\ref{efficiency}.  Finally, in
        Section~\ref{illustration} we apply our method to a few
        well-known MCMC algorithms.  Our examples include
        \pcite{albert1993bayesian} DA algorithm for Bayesian probit
        regression, and a DA algorithm for Bayesian linear regression
        with non-Gaussian errors \citep{liu1996bayesian}.
        Further application of the method can be found in \cite{zhang2019trace}.

	\section{Markov operators} \label{MarkovOperators}
	
        Assume that the Markov chain $\Phi$ has a Markov transition
        density, $p(u, \cdot),\,u \in S$, such that, for any
        measurable $A \subset S$ and $u \in S$,
	\[
	\mathbb{P}(\Phi_k \in A|\Phi_0 = u) = \int_A p^{(k)}(u,u') \,
        \mu(du') \,,
	\]
	where
	\[
	p^{(k)}(u,\cdot) := \left\{
	\begin{array}{@{}ll@{}}
	p(u,\cdot)  & k=1 \\
	\int_S p^{(k-1)}(u,u') p(u',\cdot) \, \mu(du') & k > 1 \,
	\end{array} \right. 
	\]
	is the $k$-step Mtd corresponding to $p(u,\cdot)$.  We will
        assume throughout that~$\Phi$ is Harris ergodic,
        i.e. irreducible, aperiodic and Harris recurrent. Define a
        Hilbert space consisting of complex valued functions on~$S$
        that are square integrable with respect to $\pi(\cdot),$
        namely
	\[
	L^2(\pi) := \Big\{ f:S \to \mathbb{C} \; \Big \arrowvert \,
        \int_{S} |f(u)|^2 \pi(u) \, \mu(du) < \infty \Big\} \,.
	\]
	For $f, g \in L^2(\pi),$ their inner product is given by
	\[
	\langle f, g \rangle_{\pi} = \int_{S} f(u) \overline{g(u)}
        \pi(u) \, \mu(du) \,.
	\]
	We assume that $\mathcal{U}$ is countably generated, which
        implies that $L^2(\pi)$ is separable and admits a countable
        orthonormal basis \citep[see
        e.g.][Theorem~19.2]{billingsley2008probability}. The
        transition density $p(u,\cdot), \, u\in S$ defines the
        following linear operator~$P.$ For any $f \in L^2(\pi),$
	\[
	Pf(u) = \int_{S} p(u,u') f(u') \, \mu(du') \,.
	\]
	The spectrum of a linear operator~$L$ is defined to be
	\[
	\sigma(L) = \big \{ \lambda \in \mathbb{C} \, \big | \,
        (L-\lambda I)^{-1} \text{ doesn't exist or is unbounded} \big
        \} \,,
	\]
	where~$I$ is the identity operator. It is well-known that
        $\sigma(P)$ is a closed subset of the unit disk
        in~$\mathbb{C}.$ Let $f_0 \in L^2(\pi)$ be the normalized
        constant function, i.e. $f_0(u) \equiv 1$, then $Pf_0 = f_0$.
        (This is just a fancy way of saying that 1 is an eigenvalue
        with constant eigenfunction, which is true of \textit{all}
        Markov operators defined by ergodic chains.)  Denote by $P_0$ the operator such that
        $P_0 f = Pf - \langle f, f_0 \rangle_{\pi} f_0$ for all $f \in
        L^2(\pi).$ Then the spectral gap of~$P$ is defined as
	\[
	\delta = 1 - \sup \Big\{ |\lambda| \, \Big| \, \lambda \in
        \sigma(P_0) \Big\} \,.
	\]

	For the remainder of this section, we assume that~$P$ is
        non-negative (and thus self-adjoint) and compact.  This
        implies that $\sigma(P) \subset [0,1]$, and that any
        non-vanishing element of $\sigma(P)$ is necessarily an
        eigenvalue of~$P$.  Furthermore, there are at most countably
        many eigenvalues, and they can accumulate only at the
        origin. Let $\lambda_0, \lambda_1, \dots, \lambda_{\kappa}$ be
        the decreasingly ordered strictly positive eigenvalues of~$P$
        taking into account multiplicity, where $0 \leq \kappa \leq
        \infty$. Then $\lambda_0=1$ and~$\lambda_1$ is what we
        previously referred to as the ``second largest eigenvalue'' of
        the Markov operator.  If $\kappa = 0$, we set $\lambda_1 = 0$
        (which corresponds to the trivial case where
        $\{\Phi_k\}_{k=0}^{\infty}$ are iid). Since~$\Phi$ is Harris
        ergodic, $\lambda_1$ must be strictly less than~$1$.  Also,
        the compactness of $P$ implies that of $P_0$, and it's easy to
        show that $\sigma(P_0) = \sigma(P) \backslash \{1\}$.  Hence,
        $\Phi$ is geometrically ergodic and the spectral gap is
	\[
	\delta = 1-\lambda_1 > 0 \,.
	\]
	For further background on the spectrum of a linear operator,
        see e.g. \cite{helmberg2014introduction} or
        \cite{ahues2001spectral}.

	\section{Power sums of eigenvalues} \label{powersum}
	
	We now develop some results relating $\lambda_1$ to the power
        sum of $P$'s eigenvalues.  We assume throughout this section
        that~$P$ is non-negative and trace-class (compact with
        summable eigenvalues).  For any positive integer $k$, let
	\[
	s_k = \sum_{i=0}^{\kappa} \lambda_i^k \,,
	\]
	and define $s_0$ to be infinity. The first power sum, $s_1$,
        is the trace norm of $P$ \citep[see e.g.][]{conway1990course,
          conway2000course}, while $\sqrt{s_2}$ is the Hilbert-Schmidt
        norm of~$P.$ That~$P$ is trace-class implies $s_1 < \infty,$
        and it's clear that $s_k$ is decreasing in~$k.$ 
        
     The magnitude of $s_k$ is directly related to the convergence behavior of the chain. For instance, suppose that the chain starts at a point mass $\Phi_0 = u$, then the chi-square distance between the distribution of $\Phi_k$ and the stationary distribution is given by \citep[see e.g.][]{diaconis2008gibbs}
     \[
     \chi_k^2(u) := \int_{S_U} \frac{\left(p^{(k)}(u,u') - \pi(u') \right)^2}{\pi(u')} \, \mu(du') =  \sum_{i=1}^{\kappa} \lambda_i^{2k} |f_i(u)|^2,
     \]
     where $f_i: S_U \to \mathbb{C}$ is the normalized eigenfunction corresponding to $\lambda_i$. It follows that
     \[
     s_{2k} = \sum_{i=1}^{\kappa} \lambda_i^{2k} = \int_{S_U} \chi_k^2(u) \pi(u) \, \mu(du),
     \]
     which is the average of $\chi_k^2(u)$ under~$\pi.$ More importantly, one can use functions of $s_k$ to bound $\lambda_1,$ and thus the spectral gap.
	
	Observe that,
	\begin{equation*} \lambda_1 \leq u_k := (s_k-1)^{1/k},\quad
          \forall \, k \in \mathbb{N} \,.
	\end{equation*}
	Moreover, if $\kappa \geq 1,$ then it's easy to show that
	\begin{equation*} \lambda_1 \geq l_k :=
          \frac{s_k-1}{s_{k-1}-1}, \quad \forall \, k \in \mathbb{N}
          \,.
	\end{equation*}
	We now show that, in fact, these bounds are monotone in $k$
        and converge to $\lambda_1$.
	
	\begin{proposition} \label{eigenlim}
		As $k \to \infty$,
		\begin{equation} \label{uklim}
		u_k \downarrow \lambda_1 \,,
		\end{equation}
		and if furthermore $\kappa \geq 1,$
		\begin{equation} \label{lklim}
		l_k \uparrow \lambda_1 \,.
		\end{equation}
	\end{proposition}
	\begin{proof}
          We begin with \eqref{uklim}.  When $\kappa = 0,$ $s_k \equiv
          1$ and the conclusion follows. Suppose $\kappa \geq 1,$ and
          that the second largest eigenvalue is of multiplicity $m,$
          i.e.
		\[
		1 = \lambda_0 > \lambda_1 = \lambda_2 = \dots =
                \lambda_m > \lambda_{m+1} \geq \dots \geq
                \lambda_{\kappa} > 0.
		\]
		If $\kappa = m$, then $s_k - 1 = m\lambda_1^k$ for all
                $k\geq 1$ and the proof is trivial. Suppose for the
                rest of the proof that $\kappa \geq m+1.$ For positive
                integer $k,$ let $r_k = \sum_{i=m+1}^{\kappa}
                \lambda_i^k < \infty.$ Then $r_k > 0,$ and
		\[
		\frac{r_{k+1}}{r_k} \leq \lambda_{m+1} < \lambda_1 \,.
		\]
		Hence,
		\[
		\lim\limits_{k \to \infty} \frac{r_k}{s_k-1-r_k} =
                \lim\limits_{k \to \infty} \frac{r_k}{m\lambda_1^k}
                \leq \lim\limits_{k \to \infty}
                \frac{r_1\lambda_{m+1}^{k-1}}{m\lambda_1^k} = 0 \,.
		\]
		It follows that
		\[
		\log u_k = \log \lambda_1 + \frac{1}{k}\log m +
                \frac{1}{k} \log (1+o(1)) \to \log \lambda_1 \,.
		\]
		Finally,
		\[
		u_{k+1} < \lambda_1^{1/(k+1)} \Big(
                \sum_{i=1}^{\kappa} \lambda_i^k \Big)^{1/(k+1)} \leq
                \Big( \sum_{i=1}^{\kappa} \lambda_i^k
                \Big)^{1/[k(k+1)]} \Big( \sum_{i=1}^{\kappa}
                \lambda_i^k \Big)^{1/(k+1)} = u_k \,,
		\]
		and~\eqref{uklim} follows.
		
                Now onto \eqref{lklim}.  We have already shown that
		\[
		s_k - 1 = m\lambda_1^k(1+o(1)) \,.
		\]
		Thus,
		\[
		l_k =
                \frac{m\lambda_1^k(1+o(1))}{m\lambda_1^{k-1}(1+o(1))}
                \to \lambda_1 \,.
		\]
		To show that $l_k$ is increasing in $k$, which would
                complete the proof, we only need note that
		\[
		\begin{aligned}
                  (s_{k+1}-1)(s_{k-1}-1) &= \sum_{i=1}^{\kappa} \lambda_i^{k+1} \sum_{j=1}^{\kappa} \lambda_j^{k-1} \\
                  &= \frac{1}{2}\sum_{i=1}^{\kappa} \sum_{j=1}^{\kappa} \lambda_i^{k-1}\lambda_j^{k-1} (\lambda_i^2 + \lambda_j^2) \\
                  & \geq \sum_{i=1}^{\kappa} \sum_{j=1}^{\kappa} \lambda_i^{k}\lambda_j^{k} \\
                  &= (s_k-1)^2 \,.
		\end{aligned}
		\]
	\end{proof}

	
	Suppose now that we are interested in the convergence behavior
        of a particular Markov operator that is known to be
        non-negative and trace-class.  If it is possible to estimate
        $s_k$, then Proposition~\ref{eigenlim} provides a method of
        getting approximate bounds on $\lambda_1$.  When a DA Markov
        operator is trace-class, there is a nice integral
        representation of $s_k$ that leads to a simple, classical
        Monte Carlo estimator of $s_k$.  In the following section, we
        describe some theory for DA Markov operators, and in
        Section~\ref{MC}, we develop a classical Monte Carlo estimator
        of $s_k$.

	\section{Data augmentation operators and an integral representation of $s_k$} \label{daint}
	
        In order to formally define DA, we require a second measure
        space.  Let $(S_V, \mathcal{V}, \nu)$ be a $\sigma$-finite
        measure space such that $\mathcal{V}$ is countably generated.
        Also, rename $S$ and $\pi$, $S_U$ and $\pi_U$, respectively.
        Consider the random element $(U,V)$ taking values in $S_U
        \times S_V$ with joint pdf $\pi_{U,V}(\cdot,\cdot).$ Suppose
        the marginal pdf of $U$ is the target, $\pi_U(\cdot)$, and
        denote the marginal pdf of $V$ by $\pi_V(\cdot).$ We further
        assume that the conditional densities $\pi_{U|V}(u|v) :=
        \pi_{U,V}(u,v)/\pi_V(v)$ and $\pi_{V|U}(v|u) :=
        \pi_{U,V}(u,v)/\pi_U(u)$ are well defined almost everywhere in
        $S_U \times S_V.$ Recall that~$\Phi$ is a Markov chain on the
        state space $S_U$ with Mtd $p(u,\cdot),\, u \in S_U.$ We call
        $\Phi$ a DA chain, and accordingly,~$P$ a DA operator, if
        $p(u,\cdot)$ can be expressed as
	\begin{equation} \label{DAmtd} p(u,\cdot) = \int_{S_V}
          \pi_{U|V}(\cdot|v) \pi_{V|U}(v|u) \, \nu(dv) \,.
	\end{equation}
	Such a chain is necessarily reversible with respect to $\pi_U(\cdot)$. 
	To simulate it, in each iteration, one first
        draws the latent element~$V$ using $\pi_{V|U}(\cdot|u)$, where
        $u \in S_U$ is the current state, and then given $V=v$, one
        updates the current state according to $\pi_{U|V}(\cdot|v)$.
        A DA operator is non-negative, and thus
        possesses a positive spectrum \citep{liu:wong:kong:1994}.
	
	Assume that~\eqref{DAmtd} holds. Given $k \in \mathbb{N}$, the power sum of~$P$'s eigenvalues,
        $s_k,$ if well defined, is closely related to the diagonal
        components of $p^{(k)}(\cdot,\cdot).$ Just as we can calculate
        the sum of the eigenvalues of a matrix by summing its
        diagonals, we can obtain $s_k$ by evaluating $\int_{S_U}
        p^{(k)}(u,u) \, \mu(du)$.  Here is a formal statement.
	
	\begin{theorem} \label{basic}
		The DA operator~$P$ is trace-class if and only if
		\begin{equation} \label{traceclass}
		\int_{S_U} p(u,u) \, \mu(du) < \infty \,.
		\end{equation}
		If~\eqref{traceclass} holds, then for any positive
                integer $k,$
		\begin{equation} \label{sum=int} s_k :=
                  \sum_{i=0}^{\kappa} \lambda_i^k = \int_{S_U}
                  p^{(k)}(u,u) \, \mu(du) \,.
		\end{equation}
	\end{theorem}

	Theorem~\ref{basic} is a combination of a few standard results
        in classical functional analysis.  It is fairly well-known,
        but we were unable to find a complete proof in the literature.
        An elementary proof is given in the appendix for
        completeness. For a more modern version of the theorem, see
        \cite{brislawn1988kernels}.

        It is often possible to exploit Theorem~\ref{basic} even when
        $\Phi$ is \textit{not} a DA Markov chain.  Indeed, suppose
        that $\Phi$ is reversible, but is not a DA chain.  Then $P$ is
        not a DA operator, but $P^2$, in fact, is.
        (Just take $\pi_{U,V}(u,v) = \pi_U(u)p(u,v)$.)  If, in
        addition,~$P$ is Hilbert-Schmidt, which is equivalent to
	\[
	\int_{S_U} \int_{S_U} \frac{(p(u,u'))^2 \pi_U(u)}{\pi_U(u')}
        \, \mu(du) \, \mu(du') < \infty \,,
        \]
        then by a simple spectral decomposition \citep[see e.g.][ \S 28
        Corollary 2.1]{helmberg2014introduction} one can show that
        $P^2$ is trace-class, and its eigenvalues are precisely the
        squares of the eigenvalues of~$P$.  In this case, the spectral
        gap of~$P$ can be expressed as~$1$ minus the square root of
        $P^2$'s second largest eigenvalue.  Moreover, by
        Theorem~\ref{basic}, for $k \in \mathbb{N},$ the sum of the
        $k$th power of $P^2$'s eigenvalues is equal to $\int_{S_U}
        p^{(2k)}(u,u) \, \mu(du) < \infty$.

        We now briefly describe the so-called sandwich algorithm,
        which is a variant of DA that involves an extra step
        sandwiched between the two conditional draws of DA
        \citep{liu1999parameter,hobert2008theoretical}.  Let
        $s(v,\cdot), \, v \in S_V$ be a Markov transition function
        (Mtf) with invariant density $\pi_V(\cdot)$. Then
	\begin{equation} \label{sandwich} \tilde{p}(u,\cdot) =
          \int_{S_V}\int_{S_V} \pi_{U|V}(\cdot|v') s(v,dv')
          \pi_{V|U}(v|u) \nu(dv) \,, \; u \in S_U\,,
	\end{equation}
	is an Mtd with invariant density $\pi_U(\cdot)$.  This Mtd
        defines a new Markov chain, call it $\tilde{\Phi}$, which we
        refer to as a sandwich version of the original DA chain,
        $\Phi$.  To simulate~$\tilde{\Phi}$, in each iteration, the
        latent element is first drawn from $\pi_{V|U}(\cdot|u)$, and
        then updated using $s(v,\cdot)$ before the current state is
        updated according to $\pi_{U|V}(\cdot|v')$.  Sandwich chains
        often converge much faster than their parent DA chains
        \citep[see e.g.][]{khare2011spectral}.

        Of course, $\tilde{p}(u,\cdot)$ defines a Markov operator on
        $L^2(\pi_U)$, which we refer to as~$\tilde{P}$.  It is easy to
        see that, if the Markov chain corresponding to $s(v,\cdot)$ is
        reversible with respect to $\pi_V(\cdot)$, then
        $\tilde{p}(u,\cdot)$ is reversible with respect to
        $\pi_U(\cdot)$.  Thus, when $s(v,\cdot)$ is reversible,
        $\tilde{P}^2$ is a DA operator.  Interestingly, it turns out
        that $\tilde{p}(u,\cdot)$ can often be re-expressed as the Mtd
        of a DA chain, in which case $\tilde{P}$ itself is a DA
        operator.  Indeed, a sandwich Mtd $\tilde{p}(u,\cdot)$ is said
        to be ``representable'' if there exists a random
        element~$\tilde{V}$ in $S_V$ such that
	\begin{equation} \label{representable} \tilde{p}(u,u') =
          \int_{S_V} \pi_{U|\tilde{V}}(u'|v) \pi_{\tilde{V}|U}(v|u) \,
          \nu(dv) \,,
	\end{equation}
	where $\pi_{U|\tilde{V}}(u'|v)$ and $\pi_{\tilde{V}|U}(v|u)$
        have the apparent meanings \citep[see, e.g.][]{hobe:2011}.  It
        is shown in Proposition~\ref{traceclassA} in Section~\ref{MC}
        that when~$P$ is trace-class and $\tilde{p}(u,\cdot)$ is
        representable, $\tilde{P}$ is also trace-class.  In this case,
        let $\{\tilde{\lambda}_i\}_{i=0}^{\tilde{\kappa}}$ be the
        decreasingly ordered positive eigenvalues of~$\tilde{P}$ taking into
        account multiplicity, where $0 \leq \tilde{\kappa} \leq
        \infty$.  Then $\tilde{\lambda}_0 = 1$, and $\tilde{\lambda}_1
        \leq \lambda_1 < 1$ \citep{hobert2008theoretical}.  For a
        positive integer~$k,$ we will denote
        $\sum_{i=0}^{\tilde{\kappa}} \tilde{\lambda}_i^k$ by
        $\tilde{s}_k$.  Henceforth, we assume that $\tilde{p}(u,\cdot)$
        is representable and we treat~$\tilde{P}$ as a DA operator.
	
	It follows from Theorem~\ref{basic} that in order to find
        $s_k$ or $\tilde{s}_k$, all we need to do is evaluate
        $\int_{S_U} p^{(k)}(u,u) \,\mu(du)$ or $\int_{S_U}
        \tilde{p}^{(k)}(u,u) \,\mu(du),$ where
        $\tilde{p}^{(k)}(u,\cdot)$ is the $k$-step Mtd of the sandwich
        chain.  Of course, calculating these integrals (in non-toy
        problems) is nearly always impossible, even for $k=1$.  In the
        next section, we introduce a method of estimating these two
        integrals using classical Monte Carlo.
	
	Throughout the remainder of the paper, we assume that~$P$ is a
        DA operator with Mtd given by~\eqref{DAmtd}, and
        that~\eqref{traceclass} holds.

	\section{Classical Monte Carlo} \label{MC}
	Consider the Mtd given by
	\begin{equation} \label{a} a(u, \cdot) = \int_{S_V}\int_{S_V}
          \pi_{U|V}(\cdot|v') r(v,dv') \pi_{V|U}(v|u) \, \nu(dv) \,, \; u \in S_U \,,
	\end{equation}
	where $r(v,\cdot), \, v \in S_V$ is an Mtf on $S_V$ with
        invariant pdf $\pi_V(\cdot).$ We will show in this section
        that this form has utility beyond constructing sandwich
        algorithms.  Indeed, the $k$-step Mtd of a DA algorithm (or a
        sandwich algorithm) can be re-expressed in the
        form~\eqref{a}. This motivates the development of a general
        method for estimating the integral $\int_{S_U} a(u,u) \,
        \mu(du)$, which is the main topic of this section.
	
	We begin by showing how $p^{(k)}(u,\cdot), \, u \in S_U$ can
        be written in the form~\eqref{a}.  The case $k=1$ is trivial.
        Indeed, if $r(v,\cdot)$ is taken to be the kernel of the
        identity operator, then $a(u,\cdot) = p(u,\cdot)$.  Define an
        Mtd $q(v,\cdot), \, v\in S_V$ by
	\[
	q(v,\cdot) = \int_{S_U} \pi_{V|U}(\cdot|u) \pi_{U|V}(u|v) \, \mu(du),
	\]
	and let $q^{(k)}(v,\cdot), \, k \geq 1$ denote the
        corresponding $k$-step Mtd. If we let
	\[
	r(v,dv')= q^{(k-1)}(v,v') \, \nu(dv'), \, v \in S_V
	\]
	for $k \geq 2,$ then $a(u,\cdot) = p^{(k)}(u,\cdot)$. Next,
        consider the sandwich Mtd $\tilde{p}^{(k)}(u,\cdot), \, u \in
        S_U$.  Again, the $k=1$ case is easy.  Taking
	\[
	r(v,\cdot) = s(v,\cdot)
	\]
	yields $a(u,\cdot) = \tilde{p}(u,\cdot)$.  Now let
	\[
	\tilde{q}(v,\cdot) = \int_{S_U} \int_{S_V} s(v',\cdot)
        \pi_{V|U}(v'|u) \pi_{U|V}(u|v)\, \nu(dv')\, \mu(du) \,,
	\]
	and denote the corresponding $k$-step transition function by
        $\tilde{q}^{(k)}(v,\cdot)$.  Then taking
	\[
	r(v,\cdot) = \int_{S_V} \tilde{q}^{(k-1)}(v',\cdot) s(v,dv')
	\]
	when $k \geq 2$ yields $a(u,\cdot) = \tilde{p}^{(k)}(u,\cdot).$
	

        The following proposition shows that, when $P$ is trace-class,
        $\int_{S_U} a(u,u) \, \mu(du)$ is finite.
	\begin{proposition} \label{traceclassA}
		$\int_{S_U} a(u,u) \, \mu(du) < \infty.$
	\end{proposition}
	\begin{proof}
		That $\int_{S_U} a(u,u) \, \mu(du) < \infty$ is equivalent to
		\begin{equation} \label{atraceclass0}
		\int_{S_U} \int_{S_V} \bigg( \int_{S_V} \frac{\pi_{U,V}(u,v')}{\pi_U(u)\pi_V(v')} r(v,dv') \bigg) \bigg( \frac{\pi_{U,V}(u,v)}{\pi_U(u)\pi_V(v)} \bigg) \pi_U(u) \pi_V(v) \, \nu(dv) \, \mu(du) < \infty \,.
		\end{equation}
		Note that
		\begin{equation} \label{atraceclass1}
		\int_{S_U} \bigg( \frac{\pi_{U,V}(u,v)}{\pi_U(u)\pi_V(v)} \bigg)^2 \pi_U(u) \pi_V(v) \, \mu(du)\nu(dv) = \int_{S_U} p(u,u) \, \mu(du) < \infty \,,
		\end{equation}
		and by Jensen's inequality,
		\begin{equation} \label{atraceclass2}
		\begin{aligned}
		&\int_{S_U} \int_{S_V} \bigg( \int_{S_V} \frac{\pi_{U,V}(u,v')}{\pi_U(u)\pi_V(v')} r(v,dv') \bigg)^2 \pi_U(u) \pi_V(v) \,\nu(dv) \, \mu(du) \\
		&\leq \int_{S_U} \int_{S_V} \int_{S_V} \bigg( \frac{\pi_{U,V}(u,v')}{\pi_U(u)\pi_V(v')} \bigg)^2 r(v,dv') \pi_U(u) \pi_V(v) \,\nu(dv) \, \mu(du) \\
		&= \int_{S_U} \int_{S_V} \bigg( \frac{\pi_{U,V}(u,v')}{\pi_U(u)\pi_V(v')} \bigg)^2 \pi_U(u) \pi_V(v') \,\nu(dv') \, \mu(du) \\
		&= \int_{S_U} p(u,u) \, \mu(du) \\
		& < \infty \,.
		\end{aligned}
		\end{equation}
		The inequality \eqref{atraceclass0} follows from \eqref{atraceclass1}, \eqref{atraceclass2}, and the Cauchy-Schwarz inequality.
	\end{proof}
	
        Combining Proposition~\ref{traceclassA} and
        Theorem~\ref{basic} leads to the following result: If $P$ is
        trace-class and $\tilde{p}(u,\cdot)$ is representable, then
        $\tilde{P}$ is also trace-class.  This is a generalization of
        \pcite{khare2011spectral} Theorem~1, which states that, under
        a condition on $s(v,dv')$ that implies representability, the
        trace-class-ness of $P$ implies that of $\tilde{P}$.

        We now develop a classical Monte Carlo estimator of
        $\int_{S_U} a(u,u) \, \mu(du)$.  Let $\omega:S_V \to
        [0,\infty)$ be a pdf that is almost everywhere positive.  We
        will exploit the following representation of the integral of
        interest:
	\begin{equation}\label{rep}
          \int_{S_U} a(u,u) \, \mu(du) = \int_{S_V} \int_{S_U} \bigg(\frac{\pi_{V|U}(v|u)}{\omega(v)}\bigg) \Big(\int_{S_V} \pi_{U|V}(u|v') r(v,dv') \Big) \omega(v) \, \mu(du) \,\nu(dv) \,.
	\end{equation}
	Clearly,
	\[
	\eta(u,v) := \Big(\int_{S_V} \pi_{U|V}(u|v') r(v,dv') \Big)
        \omega(v)
	\]
	defines a pdf on $S_U \times S_V$, and if $(U^*,V^*)$ has
        joint pdf $\eta(\cdot,\cdot)$, then
	\[
	\int_{S_U} a(u,u) \, \mu(du) = \mathbb{E} \bigg(
        \frac{\pi_{V|U}(V^*|U^*)}{\omega(V^*)} \bigg) \,.
	\]
        Therefore, if $\{(U^*_i,V^*_i)\}_{i=1}^{N}$ are iid random
        elements from $\eta(\cdot,\cdot)$, then
	\begin{equation} 
           \label{estimator}
           \frac{1}{N} \sum_{i=1}^{N} \frac{\pi_{V|U}(V^*_i|U^*_i)}{\omega(V^*_i)}
	\end{equation}
        is a strongly consistent and unbiased estimator of $\int_{S_U}
        a(u,u) \, \mu(du)$.  This is the Monte Carlo formula that is
        central to our discussion.

        Of course, we are mainly interested in the cases $a(u,\cdot) =
        p^{(k)}(u,\cdot)$ or $a(u,\cdot) = \tilde{p}^{(k)}(u,\cdot)$.
        We now develop algorithms for drawing from $\eta(\cdot,\cdot)$
        in these two situations.  First, assume $a(u,\cdot) =
        p^{(k)}(u,\cdot)$.  If $k=1$, then $r(u,\cdot)$ is the
        kernel of the identity operator, and
	\[
	\eta(u,v) = \pi_{U|V}(u|v) \omega(v) \,.
	\]
        If $k \ge 2$, then $r(v,dv') = q^{(k-1)}(v,v') \, dv'$, and
	\begin{equation*}
          \eta(u,v) = \Big( \int_{S_V} \pi_{U|V}(u|v') q^{(k-1)}(v,v')
          \, \nu(dv') \Big) \omega(v) = \Big( \int_{S_U} p^{(k-1)}(u',u)
          \pi_{U|V}(u'|v) \, \mu(du') \Big) \omega(v) \,.
	\end{equation*}
        Thus, when $k \ge 2$, we can draw from $\eta(u,v)$ as follows:
        Draw $V^* \sim \omega(\cdot)$, then draw $U' \sim
        \pi_{U|V}(\cdot|v^*)$, then draw $U^* \sim
        p^{(k-1)}(u',\cdot)$, and return $(u^*,v^*)$.  Of course, we
        can draw from $p^{(k-1)}(u',\cdot)$ by simply running $k-1$
        iterations of the original DA algorithm from starting value
        $u'$.  We formalize all of this in Algorithm 1.

	\baro \vspace*{2mm}
	\noindent {\rm Algorithm 1: Drawing $(U^*,V^*) \sim \eta(\cdot,\cdot)$ when
          $a(\cdot,\cdot) = p^{(k)}(\cdot,\cdot)$.}
	\begin{enumerate}
		\item Draw $V^*$ from $\omega(\cdot)$.
		\item Given $V^*=v^*,$ draw $U'$ from
                  $\pi_{U|V}(\cdot|v^*)$.
		\item If $k=1$, set $U^*=U'$.  If $k \geq 2$, given
                  $U'=u'$, draw $U^*$ from $p^{(k-1)}(u', \cdot)$ by
                  running $k-1$ iterations of the DA algorithm.
	\end{enumerate}
        \vspace*{-4mm} \barba
        \smallskip

        \noindent Similar arguments lead to the following algorithm
        for the sandwich algorithm

        \baro \vspace*{2mm}
	\noindent {\rm Algorithm 1S: Drawing $(U^*,V^*) \sim \eta(\cdot,\cdot)$ when
          $a(\cdot,\cdot) = \tilde{p}^{(k)}(\cdot,\cdot)$}
	\begin{enumerate}
		\item Draw $V^*$ from $\omega(\cdot)$.
		\item Given $V^*=v^*$, draw $V'$ from $s(v^*,\cdot)$.
		\item Given $V'=v'$ draw $U'$ from
                  $\pi_{U|V}(\cdot|v')$.
		\item If $k=1$, set $U^*=U'$.  If $k \geq 2$, given
                  $U'=u'$, draw $U^*$ from $\tilde{p}^{(k-1)}(u',
                  \cdot)$ by running $k-1$ iterations of the sandwich
                  algorithm.
	\end{enumerate}
        \vspace*{-4mm}
	\barba
	\noindent It is important to note that we do not need to know
        the representing conditionals $\pi_{U|\tilde{V}}(\cdot|v)$ and
        $\pi_{\tilde{V}|U}(\cdot|u)$ from~\eqref{representable} in
        order to run Algorithm 1S.
	
	As with all classical Monte Carlo techniques, a key element in
        successful implementation is a finite variance.  Define
	\[
	D^2 = \mathrm{var} \bigg(
        \frac{\pi_{V|U}(V^*|U^*)}{\omega(V^*)} \bigg) \,.
	\]
        Of course, $D^2 < \infty$ if and only if
	\begin{equation} \label{moment2} \int_{S_V} \int_{S_U}
          \bigg(\frac{\pi_{V|U}(v|u)}{\omega(v)}\bigg)^2 \eta(u,v) \,
          \mu(du) \,\nu(dv) < \infty \,.
	\end{equation}
	The following theorem provides a sufficient condition for
        finite variance.
	
	\begin{theorem} \label{secondmoment}
		The variance, $D^2$, is finite if
		\begin{equation} \label{momentcond}
		\int_{S_V} \int_{S_U} \frac{\pi_{V|U}^3(v|u) \pi_{U|V}(u|v)} {\omega^2(v)} \, \mu(du) \,\nu(dv) < \infty.
		\end{equation}
	\end{theorem}
	\begin{proof}
          First, note that~\eqref{moment2} is equivalent to
		\begin{equation*}
		\int_{S_V} \int_{S_U} \bigg(\frac{\pi^2_{V|U}(v|u)}{\pi_V(v)\omega(v)}\bigg) \bigg( \frac{\int_{S_V} \pi_{U|V}(u|v') r(v,dv')}{\pi_U(u)} \bigg) \pi_U(u)\pi_V(v) \, \mu(du)\,\nu(dv) < \infty.
		\end{equation*}
		Now, it follows from \eqref{momentcond} that
		\begin{equation} \label{moment2.1}
		\int_{S_V} \int_{S_U} \bigg(\frac{\pi^2_{V|U}(v|u)}{\pi_V(v)\omega(v)}\bigg)^2 \pi_U(u)\pi_V(v) \, \mu(du)\,\nu(dv) < \infty.
		\end{equation}
		Moreover, by Jensen's inequality,
		\begin{equation} \label{moment2.2}
		\begin{aligned}
		&\int_{S_V} \int_{S_U} \bigg( \frac{\int_{S_V} \pi_{U|V}(u|v') r(v,dv')}{\pi_U(u)} \bigg)^2 \pi_U(u)\pi_V(v) \, \mu(du)\,\nu(dv) \\
		& \leq \int_{S_V} \int_{S_U} \int_{S_V} \bigg( \frac{\pi_{U|V}(u|v')}{\pi_U(u)} \bigg)^2 r(v,dv') \pi_U(u)\pi_V(v) \, \mu(du)\,\nu(dv) \\
		&= \int_{S_V} \int_{S_U} \bigg( \frac{\pi_{U|V}(u|v')}{\pi_U(u)} \bigg)^2 \pi_U(u)\pi_V(v') \, \mu(du)\,\nu(dv') \\
		&= \int_{S_U} p(u,u) \, \mu(du) \\
		& < \infty.
		\end{aligned}
		\end{equation}
		The conclusion now follows from \eqref{moment2.1},
                \eqref{moment2.2}, and Cauchy-Schwarz.
	\end{proof}
	
	Theorem~\ref{secondmoment} implies that an $\omega(\cdot)$
        with heavy tails is more likely to result in finite variance
        (which is not surprising). It might seem natural to take
        $\omega(\cdot) = \pi_V(\cdot)$.  However, in practice, we are
        never able to draw from $\pi_V(\cdot)$.  (If we could do that,
        we would not need MCMC.)  
        Moreover, setting $\omega(\cdot)$ to be $\pi_V(\cdot)$ does not always result in a finite variance.
        On the other hand, it can be beneficial to use $\omega(\cdot)$s resembling $\pi_V(\cdot)$, as we argue in Section~\ref{efficiency}.

	When an appropriate $\omega(\cdot)$ is difficult to find, one can construct an alternative Monte Carlo estimator as follows.
	Let $\psi:S_U \to [0,\infty)$ be a pdf that is positive almost
        everywhere. The following dual of \eqref{rep} may also be used
        to represent $\int_{S_U}a(u,u) \, \mu(du)$:
	\begin{equation*} 
          \int_{S_U} a(u,u) \, \mu(du) = \int_{S_U}
          \int_{S_V} \int_{S_V} \frac{\pi_{U|V}(u|v)}{\psi(u)}
          r(v',dv) \pi_{V|U}(v'|u) \psi(u) \, \nu(dv') \, \mu(du) \,.
	\end{equation*}
	Now suppose that $\{(U^*_i,V^*_i)\}_{i=1}^N$ are iid from
	\[
	\zeta(u,v) \, \mu(du) \, \nu(dv) = \bigg( \int_{S_V} r(v',dv)
        \pi_{V|U}(v'|u) \, \nu(dv') \bigg) \psi(u) \, \mu(du) \,.
	\]
        The analogue of \eqref{estimator} is the following classical
        Monte Carlo estimator of $\int_{S_U} a(u,u) \, \mu(du)$:
	\begin{equation} \label{estimator2} \frac{1}{N} \sum_{i=1}^{N}
          \frac{\pi_{U|V}(U^*_i|V^*_i)}{\psi(U^*_i)} \,.
	\end{equation}
        We now state the obvious analogues of Algorithms 1 and 1S.

	\baro \vspace*{2mm}
	\noindent {\rm Algorithm 2: Drawing $(U^*,V^*) \sim \zeta(\cdot,\cdot)$ when $a(\cdot,\cdot) = p^{(k)}(\cdot,\cdot)$.}
	\begin{enumerate}
		\item Draw $U^*$ from $\psi(\cdot)$.
                \item If $k=1$, set $U'=U^*$.  If $k \geq 2$, given
                  $U^*=u^*$, draw $U'$ from $p^{(k-1)}(u^*, \cdot)$.
                \item Given $U'=u'$, draw $V^*$ from
                  $\pi_{V|U}(\cdot|u')$.
	\end{enumerate}
        \vspace*{-4mm} \barba 
        \smallskip

	\baro \vspace*{2mm}
	\noindent {\rm Algorithm 2S: Drawing $(U^*,V^*) \sim
          \zeta(\cdot,\cdot)$ when $a(\cdot,\cdot) =
          \tilde{p}^{(k)}(\cdot,\cdot)$.}
	\begin{enumerate}
		\item Draw $U^*$ from $\psi(\cdot)$.
		\item If $k=1$, set $U'=U^*$.  If $k \geq 2$, given
                  $U^*=u^*$, draw $U'$ from $\tilde{p}^{(k-1)}(u^*, \cdot)$.
                \item Given $U'=u'$, draw $V'$ from
                  $\pi_{V|U}(\cdot|u')$.
                \item Given $V'=v'$, draw $V^*$ from $s(v',\cdot)$.
	\end{enumerate}
        \vspace*{-4mm} \barba

	Let $D'^2$ be the variance of $\pi_{U|V}(U^*|V^*)/\psi(U^*)$ under~$\zeta$. To ensure that it's finite,
        we only need
	\begin{equation} \label{moment2dual} \int_{S_U} \int_{S_V}
          \int_{S_V} \bigg(\frac{\pi_{U|V}(u|v)}{\psi(u)}\bigg)^2
          r(v',dv) \pi_{V|U}(v'|u) \psi(u) \,\nu(dv') \, \mu(du) <
          \infty \,.
	\end{equation}
	The following result is the analogue of Theorem~\ref{secondmoment}.
	\begin{corollary} \label{secondmomentdual} The variance, $D'^2$, is finite if
          \begin{equation} \label{momentconddual} \int_{S_U}
            \int_{S_V} \frac{\pi_{U|V}^3(u|v) \pi_{V|U}(v|u)}
            {\psi^2(u)} \,\nu(dv) \, \mu(du) < \infty \,.
		\end{equation}
	\end{corollary}
	\begin{proof}
          Note that the left hand side of~\eqref{moment2dual} is equal
          to
		\[
		\int_{S_U} \int_{S_V} \bigg(\int_{S_V}
                \frac{\pi_{U|V}^2(u|v)}{\psi(u)\pi_U(u)} r(v',dv)
                \bigg) \bigg( \frac{\pi_{V|U}(v'|u)}{\pi_V(v')} \bigg)
                \pi_U(u)\pi_V(v') \,\nu(dv') \, \mu(du) \,.
		\]
		Apply the Cauchy-Schwarz inequality, then utilize
                Jensen's inequality to get rid of $r(v',dv),$ and
                finally make use of~\eqref{momentconddual} and the
                fact that $P$ is trace-class.
	\end{proof}

	Typically, it's easy to select a good sampling density $\omega(\cdot)$ for Algorithm~1 when $S_V$ is low dimensional, or to select a good $\psi(\cdot)$ for Algorithm~2 when $S_U$ is low dimensional.
	For DA algorithms used in Bayesian models, it's often the case that $\mbox{dim} (S_U) = p$, and $\mbox{dim} (S_V) = n$, where~$p$ and~$n$ are, respectively, the number of unknown parameters in the model and the number of observations.
	When this is the case, the estimator~\eqref{estimator} is likely to be efficient when~$n$ is small, while~\eqref{estimator2} is likely to be efficient when~$p$ is small.

	Suppose that we have obtained estimates of $s_k$ and $s_{k-1}$ based on~\eqref{estimator} or~\eqref{estimator2}, call them $s^*_k$ and $s^*_{k-1}$. Then $u^*_k = (s^*_k-1)^{1/k}$ and $l^*_k = (s^*_k-1)/(s^*_{k-1}-1)$ serve as point estimates of $u_k$ and $l_k$, respectively. When our estimators have finite variances, we can acquire, via the delta method, confidence intervals for $u_k$ and $l_k$. Assume that a confidence interval for $l_k$ is $(a_k,b_k)$ and a confidence interval for $u_k$ is $(c_k, d_k)$, then $(a_k, d_k)$ is an interval estimate for $\lambda_1$. Interval estimates of $\tilde{\lambda}_1$ can be derived in a similar fashion.
	
	It's worth pointing out that $u_k$ is a nontrivial upper bound on $\lambda_1 \in [0,1)$ only if $s_k < 2$. 
	The parameter~$k$ can be determined sequentially.
	Take Algorithm~1 for example. 
	Suppose that we have drawn~$N$ iid copies of $(U^*, V^*)$ from $\eta(\cdot,\cdot)$ with $a(\cdot, \cdot) = p^{(k)}(\cdot,\cdot)$, but find that $s^*_k$ is not small enough for our purposes.
	Since $s_k$ is decreasing in~$k$, we wish to increase~$k$ by a positive integer~$j$.
	To draw $(U^{**}, V^{**})$ from $\eta(\cdot,\cdot)$ with $a(\cdot, \cdot) = p^{(k+j)}(\cdot,\cdot)$, we only need to set $V^{**} = V^*$, and draw $U^{**}$ from $p^{(j)}(U^*, \cdot)$. 
	This procedure can be repeated until the estimated power sum $s^*_{k+j}$ is decreased to a satisfactory value.
	More guidance on the choice of~$k$ can be found in the next section.

	\section{Efficiency of the algorithm} \label{efficiency}
	
	To obtain an interval estimate of $\lambda_1$ based on~\eqref{estimator} or~\eqref{estimator2}, one needs to run~$N$ iterations of Algorithm~1 or~2. If the time needed to simulate one step of the DA chain is $\tau$, then the time needed to run $N$ iterations of Algorithm 1 or 2 is approximately $kN\tau$. Note that significant speedup can be achieved through parallel computing, since the~$N$ iterations are carried out independently. Given $k$ and $N$, the accuracy of the estimate depends on two factors:
	1. The distance between $l_k$ and $u_k$, and 2. The errors in the estimates, $l^*_k$ and $u^*_k$.  We now briefly analyze these two factors, and give some additional guidelines regarding the choice of $\omega(\cdot)$ and $\psi(\cdot)$.
	
	As before, suppose that
	\[
	1 = \lambda_0 > \lambda_1 = \lambda_2 = \cdots = \lambda_m > \lambda_{m+1} \geq \cdots \geq \lambda_{\kappa} > 0
	\]
	for some $m < \infty$. Clearly, as $k \to \infty$,
	\[
	s_k - 1 = \lambda_1^k \left( m + O \left( \lambda_{m+1}^k/\lambda_1^k \right) \right).
	\]
	Hence, as $k \to \infty$,
	\[
	l_k := \frac{s_k-1}{s_{k-1}-1} = \lambda_1 \left( 1 + O \left( \lambda_{m+1}^{k-1}/\lambda_1^{k-1} \right) \right),
	\]
	and
	\[
	\begin{aligned}
	u_k &:= (s_k-1)^{1/k}\\
	&= \lambda_1 m^{1/k} \left(1 + O\left( k^{-1} \lambda_{m+1}^k/\lambda_1^k \right) \right) \\
	&= \lambda_1 \left( 1 + (\log m) O\left(k^{-1}\right) \right) \left(1 + O\left( k^{-1} \lambda_{m+1}^k/\lambda_1^k \right) \right).\\
	&= \left\{
	\begin{array}{@{}ll@{}}
	\lambda_1 \left(1 + O\left( k^{-1} \lambda_2^k/\lambda_1^k \right) \right)  & m=1 \\
	\lambda_1 \left( 1 + O\left(k^{-1}\right) \right) & m > 1. \,
	\end{array} \right. 
	\end{aligned}
	\]
	Depending on whether $m=1$ or not, $u_k-l_k$ decreases at either a geometric or polynomial rate as~$k$ grows.
	
	The errors of $l^*_k$ and $u^*_k$ arise from those of $s^*_k$ and $s^*_{k-1}$. We now consider the estimator~\eqref{estimator2} for estimating $s_k$. Its variance is given by
	\[
	\begin{aligned}
	\frac{D'^2}{N} &= \frac{1}{N} \left\{ \int_{S_U}\int_{S_V}\int_{S_V} \bigg(\frac{\pi_{U|V}(u|v)}{\psi(u)}\bigg)^2
	r(v',dv) \pi_{V|U}(v'|u) \psi(u) \,\nu(dv') \, \mu(du) - s_k^2 \right\}\\
	&= \frac{1}{N} \left\{ \int_{S_U}\int_{S_V} \int_{S_U} \frac{\pi_{U|V}^2(u|v)}{\psi(u)} \pi_{V|U}(v|u') p^{(k)}(u,u') \, \mu(du') \, \nu(dv) \, \mu(du) - s_k^2 \right\}.
	\end{aligned}
	\]
	Note that
	\[
	p_k\left( (u,v), (u',v') \right) := \pi_{V|U}(v'|u') p^{(k)}(u,u')
	\]
	gives the $k$-step Mtd of a Gibbs chain whose stationary pdf is $\pi_{U,V}(\cdot,\cdot)$. Thus, under suitable conditions, for almost any $u \in S_U$,
	\[
	\begin{aligned}
	\lim\limits_{k \to \infty} s_k(u) &:= \lim\limits_{k \to \infty} \int_{S_V} \int_{S_U} \frac{\pi_{U|V}^2(u|v)}{\psi(u)} \pi_{V|U}(v|u') p^{(k)}(u,u') \, \mu(du') \, \nu(dv) \\
	&= \int_{S_V} \int_{S_U} \frac{\pi_{U|V}^2(u|v)}{\psi(u)} \pi_{U,V}(u',v) \, \mu(du') \, \nu(dv) \\
	&= \frac{p(u,u) \pi_U(u)}{\psi(u)}.
	\end{aligned}
	\]
	As $k \to \infty$, we expect
	\[
	D'^2 = \int_{S_U} s_k(u) \, \mu(du) - s_k^2 \to \int_{S_U} \frac{p(u,u) \pi_U(u)}{\psi(u)} \, \mu(du) - 1.
	\]
	Suppose that $\psi(u) \approx \pi_U(u)$, then heuristically,
	\[
	\int_{S_U} \frac{p(u,u) \pi_U(u)}{\psi(u)} \, \mu(du) - 1 \approx \int_{S_U} p(u,u) \mu(du) - 1 = s_1 - 1.
	\]
	Thus, if the sum of~$P$'s eigenvalues, $s_1$, is relatively small, we recommend picking $\psi(\cdot)$s that resemble $\pi_U(\cdot)$, with possibly heavier tails (to ensure that the moment condition~\eqref{momentconddual} holds). By a similar argument, when using the estimator~\eqref{estimator}, picking $\omega(\cdot)$s that resemble $\pi_V(\cdot)$ is likely to control $D^2$ around $s_1 - 1$ for large~$k$s.
	
	While (under suitable conditions) the variance of $s_k^*$ converges to a constant as $k \rightarrow \infty$, this is not the case for $u^*_k$ and $l^*_k$ (because $u_k$ and $l_k$ are non-linear in $s_k$ and $s_{k-1}$). In fact, using the delta method, one can show that these variances are unbounded. Thus, there's a trade-off between decreasing $u_k - l_k$ (by increasing~$k$) and controlling the errors of $u^*_k$ and $l^*_k$. We do not recommend increasing~$k$ indefinitely. As long as~$k$ is large enough so that $s_k - 1$ is significantly smaller than~$1$, $u_k$ serves as a non-trivial (and often decent) upper bound for $\lambda_1$.

	\section{Examples} \label{illustration} In this section,
        we apply our Monte Carlo technique to several common Markov
        operators.  In particular, we examine one toy Markov chain,
        and two practically relevant Monte Carlo Markov chains.  In
        the two real examples, we are able to take advantage of
        existing trace-class proofs to establish
        that~\eqref{momentcond} (or~\eqref{momentconddual}) hold for
        suitable $\omega(\cdot)$ (or $\psi(\cdot)$).
	
	\subsection{Gaussian chain}
	We begin with a toy example.  Let $S_U=S_V=\mathbb{R},$
        $\pi_U(u) \propto \exp(-u^2),$ and
	\[
	\pi_{V|U}(v|u) \propto \exp \Big\{ -4\Big(v-\frac{u}{2}\Big)^2
        \Big\} \,.
	\]
	Then
	\[
	\pi_{U|V} (u|v) \propto \exp \{ -2(u-v)^2 \} \,.
	\]
	This leads to one of the simplest DA chains known.  Indeed,
        the Mtd,
	\[
	p(u,\cdot) = \int_{\mathbb{R}} \pi_{U|V}(\cdot|v)
        \pi_{V|U}(v|u) \, dv \,, \; u \in S_U \,,
	\]
	can be evaluated in closed form, and turns out to be a normal
        pdf.  The spectrum of the corresponding Markov operator, $P$,
        has been studied thoroughly \citep[see
        e.g.][]{diaconis2008gibbs}.  It is easy to verify
        that~\eqref{traceclass} holds, so~$P$ is trace-class.  In
        fact, $\kappa = \infty$, and for any non-negative integer~$i,$
        $\lambda_i = 1/2^i$.  Thus, the second largest eigenvalue,
        $\lambda_1$, and the spectral gap,~$\delta$, are both equal to
        $1/2$.  Moreover, for any positive integer~$k$,
	\[
	s_k = \sum_{i=0}^{\infty} \frac{1}{2^{ik}} =
        \frac{1}{1-2^{-k}} \,.
	\]
	
        We now pretend to be unaware of this spectral information, and
        use \eqref{estimator} to estimate
        $\{s_k,l_k,u_k\}_{k=1}^4$.  Recall that $l_k$ and $u_k$ are
        lower and upper bounds for $\lambda_1$, respectively.  Note
        that
	\[
	\int_{\mathbb{R}} \pi^3_{V|U}(v|u) \pi_{U|V} (u|v) \, du
        \propto \exp \Big( -\frac{6}{5}v^2 \Big) \,.
	\]
        It follows that, if we take $\omega(v) \propto \exp(-v^2/2)$,
        then \eqref{momentcond} holds, and our estimator of $s_k$ has
        finite variance.  We use a Monte Carlo sample size of $N = 1
        \times 10^5$ to form our estimates, and the results are shown
        in Table~\ref{gaussian}.
        
        \begin{table}[!h]
           \begin{center} \caption{Estimated power sums of eigenvalues for the Gaussian chain} \label{gaussian}
             \medskip
			\begin{tabular}{ccccc} 
                          $k$ & Est. $s_k$ & Est. $D/\sqrt{N}$ & Est. $l_k$ & Est. $u_k$ \\ \hline
                          $1$ & $1.996$ & $0.004$ & $0.000$ & $0.996$ \\
                          $2$ & $1.331$ & $0.004$ & $0.333$ & $0.575$ \\
                          $3$ & $1.142$ & $0.004$ & $0.429$ & $0.522$ \\
                          $4$ & $1.068$ & $0.004$ & $0.482$ & $0.511$ \\
			\end{tabular}
            \end{center}
         \end{table}

         Note that the estimates of the $s_k$s are quite good.  We then
         construct $95\%$ confidence intervals (CIs) for $l_4$ and
         $u_4$ via the delta method, and the results are
         $(0.442,0.522)$ and $(0.498,0.524)$, respectively.
         
         \begin{figure}
         	\centering
         	\begin{subfigure}[t]{0.3\textwidth}
         		\includegraphics[width=\textwidth]{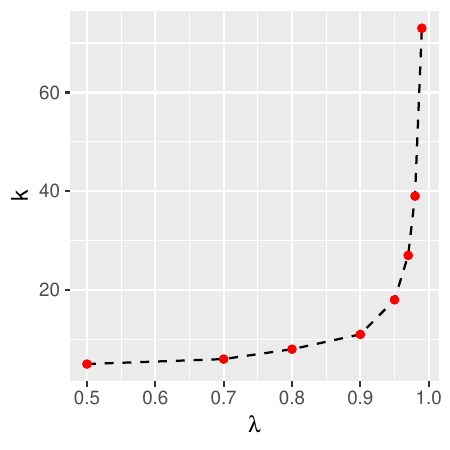}
         		\caption{The~$k$s used for different values of~$\lambda$.	
         	} \label{fig:ClosingSpectralGap0}
         	\end{subfigure}
         	\hspace{0.5cm}
         	\begin{subfigure}[t]{0.4\textwidth}
         		\includegraphics[width=\textwidth]{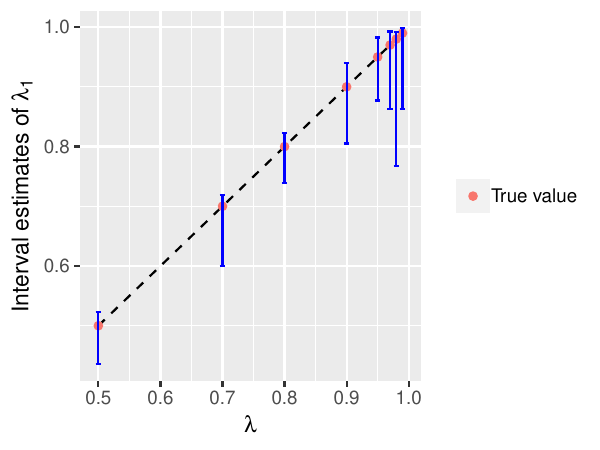}
         		\caption{Interval estimates for $\lambda_1$.} \label{fig:ClosingSpectralGap}
         	\end{subfigure}
         	\caption{Spectral gap estimation for the Gaussian chain for different~$\lambda$s}
         \end{figure}
         
         We now add an additional parameter to our toy example in
         order to study the effect of a closing spectral gap on our method.
         In particular, let
         $\pi_{V|U}(\cdot|u)$, $u \in \mathbb{R}$, be the pdf of
         $\mbox{N}(\lambda u, \lambda(1-\lambda)/2)$, where $\lambda \in (0,1)$.
         Note that our original example corresponds to $\lambda=1/2$.
         The eigenvalues of the resultant DA operator are $\{\lambda_i\}_{i=0}^{\infty} = \{\lambda^i\}_{i=0}^{\infty}$.
         We investigate the effectiveness of our method as $\lambda_1 = \lambda$ goes to~$1$, that is, as the spectral gap $\delta = 1-\lambda$ closes.
         To this end, consider a sequence of Gaussian chains with~$\lambda$ increasing from $0.5$ to $0.99$.
         In accordance with the discussion in Section~\ref{efficiency}, for a given~$\lambda$, we set~$\omega$ to be the density function of a $t$-distribution with similar variance as $\pi_V(\cdot)$, which is the pdf of $\mbox{N}(0,\lambda/2)$.
         One can verify that \eqref{momentcond} holds for every $\lambda \in (0,1)$.
         Note that in order for $u_k = (s_k-1)^{1/k}$ to be a non-trivial upper bound on $\lambda_1$, we need $s_k < 2$.
         As~$\lambda$ increases, so does $s_k$ for any given~$k$, and thus one must increase~$k$ in order to find a useful upper bound.
         Figure~\ref{fig:ClosingSpectralGap0} shows the~$k$s used for different $\lambda$s.
         When $\lambda = 0.5$, we only need $k=4$ to get a decent result; but when $\lambda = 0.99$, $k \approx 70$ is needed.
         Recall that the time needed to run~$N$ iterations of Algorithm~1 is approximately $kN\tau$, where~$\tau$ is the time needed to simulate one step of the DA chain, which is roughly the same for any $\lambda \in (0,1)$.
         To compare the performance of our method for different~$\lambda$s, we fix $kN = 1 \times 10^6$, and compare the length of the interval estimates of $\lambda_1$.
         The results are shown in Figure~\ref{fig:ClosingSpectralGap}.
         As~$\lambda$ grows, so does~$k$, and we are forced to use a smaller sample size~$N$.
         Thus, as~$\lambda$ grows, it becomes more difficult to estimate the variances of $u_k^*$ and $l_k^*$ accurately.
         As a result, the length of the interval estimate of $\lambda_1$ becomes less stable when~$\lambda$ is near~$1$.
         This is reflected in Figure~\ref{fig:ClosingSpectralGap} by an unusually wide interval estimate at $\lambda = 0.97$.
         On the other hand, most of the interval estimates at other values of~$\lambda$ near~$1$ are reasonably well-behaved.

	\subsection{Bayesian probit regression}
	
	Let $Y_1,Y_2,\dots, Y_n$ be independent Bernoulli random
        variables with $\mathbb{P}(Y_1=1|\beta) = \Phi(x_i^T\beta),$ where $x_i,\beta
        \in \mathbb{R}^p$, and $\Phi(\cdot)$ is the cumulative distribution function of the standard normal distribution.  Take the prior on $\beta$ to be
        $\mbox{N}_p(Q^{-1}w, Q^{-1}),$ where $w \in \mathbb{R}^p$ and $Q$ is
        positive definite. The resulting posterior distribution is
        intractable, but \cite{albert1993bayesian} devised a DA
        algorithm to sample from it. Let $z=(z_1,z_2,\dots,z_n)^T$ be
        a vector of latent variables, and let $X$ be the design matrix
        whose $i$th row is $x_i^T.$ The Mtd of the Albert and Chib
        (AC) chain, $p(\beta,\cdot), \beta \in \mathbb{R}^p,$ is
        characterized by
	\[
        \pi_{U|V} (\beta|z) \propto \exp \bigg[ -\frac{1}{2} \big\{
        \beta - \big( X^TX+Q \big)^{-1} \big( w+X^Tz \big) \big\}^T
        \big( X^TX + Q \big) \big\{ \beta - \big( X^TX+Q \big)^{-1}
        \big( w+X^Tz \big) \big\} \bigg] \,,
        \]
        and
        \[
        \pi_{V|U}(z|\beta) \propto \prod_{i=1}^{n}\exp \bigg\{
        -\frac{1}{2}\big(z_i-x_i^T\beta\big)^2 \bigg\}
        I_{\mathbb{R}_+}\big((y_i-0.5)z_i \big) \,.
        \]
	The first conditional density, $\pi_{U|V}(\cdot|z)$, is a
        multivariate normal density, and the second conditional
        density, $\pi_{V|U}(\cdot|\beta)$, is a product of univariate
        truncated normal pdfs.  
        
        A sandwich step can be added to facilitate the convergence of
        the AC chain. \cite{chakraborty2016convergence} constructed a
        Haar PX-DA variant of the chain, which is a sandwich chain
        with transition density of the form~\eqref{sandwich} (see also
        \cite{roy2007convergence}). The sandwich step $s(z, dz')$ is
        equivalent to the following update: $z \mapsto z' = gz$, where
        the scalar~$g$ is drawn from the following density:
        \[
        \pi_G(g|z) \propto g^{n-1} \exp \bigg[ -\frac{1}{2} z^T\big\{ I_n - X(X^TX+Q)^{-1}X^T \big\} z g^2 + z^TX(X^TX + Q)^{-1}w g   \bigg].
        \]
        Note that this pdf is particularly easy to sample from when $w = 0$.
        
        \citet{chakraborty2016convergence}
        showed that, for the AC chain,~$P$ is trace-class when one uses a concentrated
        prior (corresponding to $Q$ having large eigenvalues). In
        fact, the following is shown to hold in their proof.
	\begin{proposition} \label{probittrace} Suppose that~$X$ is full
          rank. If all the eigenvalues of $Q^{-1/2}X^TXQ^{-1/2}$ are less
          than $7/2,$ then for any polynomial function $t:
          \mathbb{R}^p \to \mathbb{R}$,
		\[
		\int_{\mathbb{R}^p} |t(\beta)| p(\beta,\beta) \,
                d\beta < \infty \,.
		\]
	\end{proposition}
	We will use the estimator \eqref{estimator2}.  The following
        proposition provides a class of $\psi(\cdot)$s that lead to
        estimators with finite variance.
	\begin{proposition} \label{probitpsi} Suppose the hypothesis
          in Proposition~\ref{probittrace} holds. If $\psi(\cdot)$ is
          the pdf of a $p$-variate $t$-distribution, i.e.
		\[
		\psi(\beta) \propto \bigg\{ 1 + \frac{1}{a}(\beta - b)^T \Sigma^{-1}  (\beta - b) \bigg\}^{-(a+p)/2}
		\]
		for some $b \in \mathbb{R}^p,$ positive definite
                matrix $\Sigma \in \mathbb{R}^{p \times p},$ and
                positive integer~$a,$ then the
                estimator~\eqref{estimator2} has finite variance.
	\end{proposition}
	\begin{proof}
          Note that for every $\beta$ and $z$
		\[
		\pi^3_{U|V} (\beta|z) \leq C \pi_{U|V} (\beta|z) \,,
		\]
		where $C$ is a constant. Hence, by Proposition~\ref{probittrace}, for any polynomial
                function $t: \mathbb{R}^p \to \mathbb{R},$
		\[
		\int_{\mathbb{R}^p} \int_{\mathbb{R}^n} |t(\beta)|
                \pi^3_{U|V} (\beta|z) \pi_{V|U} (z|\beta) \,
                dz\,d\beta \leq C\int_{\mathbb{R}^p} |t(\beta)|
                p(\beta,\beta) \, d\beta < \infty.
		\]
		Since $\psi^{-2}(\cdot)$ is a polynomial function on
                $\mathbb{R}^p$, the moment
                condition~\eqref{momentconddual} holds. The result
                follows from Corollary~\ref{secondmomentdual}.
	\end{proof}
	
        As a numerical illustration, we apply our method to the Markov
        operator associated with the AC chain corresponding to the
        famous ``lupus data'' of \cite{van2001art}.  In this dataset,
        $n=55$ and $p=3$.  
        We will construct an
        asymptotically valid 95\% CI for the second largest eigenvalue, and
        this appears to be the most rigorous and detailed analysis to date of
        the spectrum of a practically relevant MCMC algorithm on an
        uncountable state space.
        As in \cite{chakraborty2016convergence}, we
        will let $w=0$ and $Q=X^TX/c$, where $c=3.499999$.  
        It can be easily
        shown that the assumptions in Proposition~\ref{probittrace}
        are met.  \cite{chakraborty2016convergence} compared the AC
        chain, $\Phi$, and its Haar PX-DA variant,~$\tilde{\Phi}$, defined a few paragraphs ago. This comparison
        was done using estimated autocorrelations.  Their results
        suggest that~$\tilde{\Phi}$ outperforms~$\Phi$
        when estimating a certain test function.  We go further and
        estimate the second largest eigenvalue of each operator.

	It can be shown that the posterior pdf, $\pi_{U}(\cdot)$, is log-concave, and thus possess a unique mode. Let $\hat{\beta}$ be the posterior mode, and $\hat{\Sigma}$
        the estimated variance of the MLE. We pick $\psi(\cdot)$ to
        be the pdf of $t_{30}(\hat{\beta},(\hat{\Sigma}^{-1}+Q)^{-1}).$ This is to say,
        for any $\beta \in \mathbb{R}^p,$
	\[
	\psi(\beta) \propto \Big\{ 1 + \frac{1}{30} (\beta-\hat{\beta})^T (\hat{\Sigma}^{-1}+Q) (\beta-\hat{\beta}) \Big\}^{-(p+30)/2}.
	\]
	By Proposition~\ref{probitpsi}, this choice of $\psi(\cdot)$
        guarantees finite variance. When~$n$ is large, $\psi(\cdot)$ is expected to resemble $\pi_U(\cdot)$. The performance of our method seems insensitive to the degrees of freedom of the $t$-distribution (which is set at 30 for illustration).
	
        We use a Monte Carlo sample size of $N = 4\times 10^5$ to
        form our estimates for the DA operator, and the results are
        shown in Table~\ref{probit}.  Asymptotic $95\%$ CIs for $l_5$
        and $u_5$ are $(0.397,0.545)$ and $(0.573,0.595)$,
        respectively.  Using a Bonferroni argument, we can state that
        asymptotically, with at least $95\%$ confidence, $\lambda_1
        \in (0.397, 0.595)$.

	\begin{table}[!h]
		\begin{center} \caption{Estimated power sums of eigenvalues for the AC chain} \label{probit}
             \medskip
			\begin{tabular}{ccccc} 
				$k$ & Est. $s_k$ & Est. $D'/\sqrt{N}$ & Est. $l_k$ & Est. $u_k$ \\ \hline
				$1$ & $6.744$ & $0.072$ & $0.000$ & $5.744$ \\
				$2$ & $2.041$ & $0.007$ & $0.181$ & $1.020$ \\
				$3$ & $1.363$ & $0.004$ & $0.349$ & $0.713$ \\
				$4$ & $1.156$ & $0.004$ & $0.430$ & $0.628$ \\
				$5$ & $1.068$ & $0.003$ & $0.436$ & $0.584$ \\
			\end{tabular}
		\end{center}
	\end{table}

	\begin{table}[!h]
		\begin{center} \caption{Estimated power sums of eigenvalues for the Haar PX-DA version of the AC chain} \label{probit2}
			\medskip
			\begin{tabular}{ccccc} 
				$k$ & Est. $\tilde{s}_k$ & Est. $D'/\sqrt{N}$ & Est. $\tilde{l}_k$ & Est. $\tilde{u}_k$ \\ \hline
				$1$ & $3.796$ & $0.012$ & $0.000$ & $1.796$ \\
				$2$ & $1.538$ & $0.004$ & $0.193$ & $0.734$ \\
				$3$ & $1.172$ & $0.004$ & $0.319$ & $0.556$ \\
				$4$ & $1.060$ & $0.003$ & $0.352$ & $0.496$ \\
				$5$ & $1.025$ & $0.003$ & $0.419$ & $0.479$ \\
			\end{tabular}
		\end{center}
	\end{table}
	
	We now consider the sandwich chain, $\tilde{\Phi}$.  It is known that the Mtd of any Haar PX-DA chain is representable
        \citep{hobert2008theoretical}.  Hence,~$\tilde{P}$ is indeed a
        DA operator.  Recall that
        $\{\tilde{\lambda}_i\}_{i=0}^{\tilde{\kappa}}, \, 0 \leq
        \tilde{\kappa} \leq \infty$, denote the decreasingly ordered
        positive eigenvalues of $\tilde{P}$.  It was shown in
        \cite{khare2011spectral} that $\tilde{\lambda}_i \leq
        \lambda_i$ for $i \in \mathbb{N}$ with at least one strict
        inequality.  For a positive integer~$k,$ 
        $\sum_{i=0}^{\tilde{\kappa}}
        \tilde{\lambda}_i^k$ is denoted by $\tilde{s}_k$.  Let $\tilde{u}_k$ and $\tilde{l}_k,$ be the respective counterparts of $u_k$ and $l_k$.  Estimates of $\tilde{s}_k,\,k=1,2,\cdots,5$ using
        $4\times 10^5$ Monte Carlo samples are given in
        Table~\ref{probit2}.  Our estimate of $\tilde{s}_1-1$ is less
        than half of $s_1-1$, implying that, in an average sense, the
        sandwich version of the AC chain reduces the nontrivial
        eigenvalues of $P$ by more than half.  Asymptotic $95\%$ CIs
        for $\tilde{l}_5$ and $\tilde{u}_5$ are $(0.321, 0.518)$ and $(0.456,0.503)$.
        Thus, asymptotically, with at least $95\%$ confidence,
        $\tilde{\lambda}_1 \in (0.321,0.503)$. The method does not detect a significant difference between $\lambda_1$ and $\tilde{\lambda}_1$.
        
        \begin{figure}
        	\centering
        	\begin{subfigure}[t]{0.4\textwidth}
        		\includegraphics[width=\textwidth]{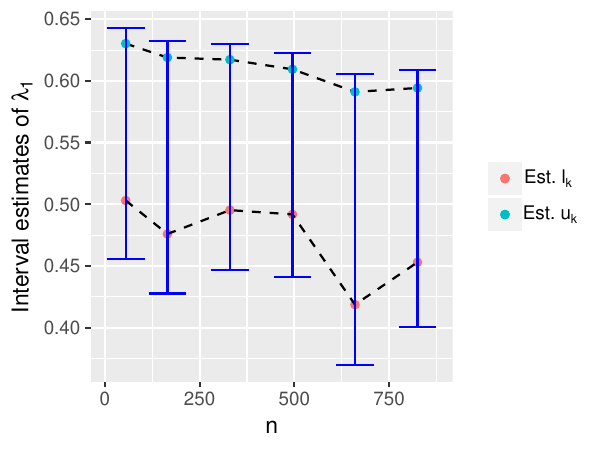}
        		\caption{Interval estimates for $\lambda_1$ for different~$n$s.} \label{fig:ProbitGrowingN}
        	\end{subfigure}
        	\hspace{0.5cm}
        	\begin{subfigure}[t]{0.4\textwidth}
        		\includegraphics[width=\textwidth]{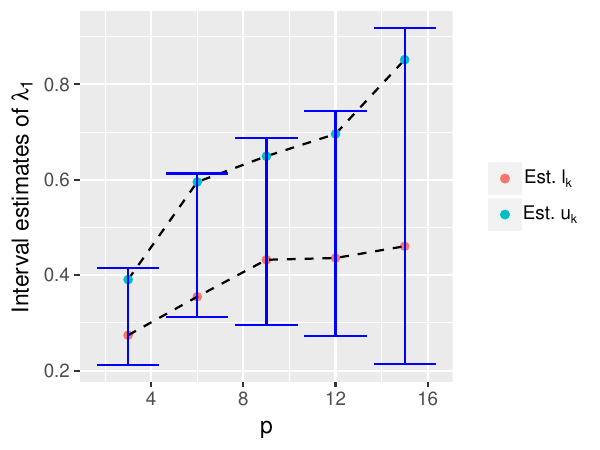}
        		\caption{Interval estimates for $\lambda_1$ for different~$p$s.} \label{fig:ProbitGrowingP}
        	\end{subfigure}
        	\caption{Spectral gap estimation for the AC chain}
        \end{figure}
	
	We now study the performance of our method when~$n$ or~$p$ increases for the original AC chain.
	First, consider a sequence of datasets where~$n$ grows.
	Let $X_{\scriptsize\mbox{lupus}} \in \mathbb{R}^{55 \times 3}$ be the design matrix for the lupus data, and let~$r$ be a positive integer.
	Set $X \in \mathbb{R}^{n \times 3}$ to be~$r$ copies of $X_{\scriptsize\mbox{lupus}}$ stacked on top of each other, so that $n=55r$.
	The response vector $(Y_1,Y_2,\dots,Y_n)^T$ is randomly generated in accordance with the probit regression model with the true value of~$\beta$ being $(-3,0,3)^T$.
	Let~$r$ range from~$1$ to $15$.
	This gives rise to a sequence of datasets with~$n$ growing from $55$ to $825$.
	An interval estimate for $\lambda_1$ is then constructed for each of these datasets.
	Throughout the simulation,~$k$ is fixed at~$5$, and~$N$ is fixed at $4 \times 10^5$.
	The result is given in Figure~\ref{fig:ProbitGrowingN}.
	Increasing~$n$, which is the dimension of $S_V$, apparently does not undermine the method.
	
	Now we consider a sequence of datasets where~$n$ is fixed and~$p$ grows.
	Let $n=200$, and let~$X$ be a $200 \times p$ matrix whose $ij$th element is $p_j(i)$, with $\{p_j(\cdot)\}_{j=1}^p$ being a set of orthogonal polynomials generated using the \verb|R| function \verb|poly()|.
	The response vector is randomly generated according to the probit model with the true value of~$\beta$ being $(-p, -p+2p/(1-p), -p + 4p/(1-p), \cdots, p)^T$.
	We apply our method to such a dataset when~$p$ is increased from~$3$ to $15$. 
	$N$ is set to be $4\times 10^5$, and~$k$ is either~$4$ or~$5$, whichever yields a better estimate.
	The interval estimates for $\lambda_1$ are given in Figure~\ref{fig:ProbitGrowingP}.
	As~$p$ increases, the length of the interval estimate grows quite rapidly, indicating that the method does not scale well with~$p$, that is, the dimension of~$S_U$.
	This is consistent with the analysis near the end of Section~\ref{MC}, which suggests that Algorithm~2 works well when~$n$ is large and~$p$ is small, but not the other way around.

	\subsection{Bayesian linear regression model with non-Gaussian errors}
	
	Let $Y_1,Y_2,\dots,Y_n$ be independent $d$-dimensional random
	vectors from the linear regression model
	\[
	Y_i = \beta^T x_i + \Sigma^{1/2}\varepsilon_i \,,
	\]
	where $x_i \in \mathbb{R}^{p}$ is known, while $\beta \in
	\mathbb{R}^{p \times d}$ and the $d\times d$ positive definite
	matrix $\Sigma$ are to be estimated.  The iid errors,
	$\varepsilon_1,\varepsilon_2,\dots,\varepsilon_n$, are assumed
	to have a pdf that is a scale mixture of Gaussian densities:
	\[
	f_h(\varepsilon) = \int_{\mathbb{R}_+} \frac{u^{d/2}}{(2\pi)^{d/2}} \exp \Big( -\frac{u}{2} \varepsilon^T\varepsilon \Big) h(u) \, du,
	\]
	where $h(\cdot)$ is a pdf with positive support, and
	$\mathbb{R}_+ := (0,\infty).$ For instance, if $d=1$ and $h(u)
	\propto u^{-2}e^{-1/(8u)},$ then $\varepsilon_1$ has pdf
	proportional to $e^{-|\varepsilon|/2}.$
	
	To perform a Bayesian analysis, we require a prior on the
	unknown parameter, $(\beta,\Sigma)$.  We adopt the (improper)
	Jeffreys prior, given by $1/|\Sigma|^{(d+1)/2}$.  Let $y$
	represent the $n \times d$ matrix whose $i$th row is the
	observed value of $Y_i$.  The following four conditions, which
	are sufficient for the resulting posterior to be proper
	\citep{qin2018trace, fernandez1999multivariate}, will be
	assumed to hold:
	\begin{enumerate}
		\item $n \geq p + d$,
		\item $(X:y)$ is full rank, where $X$ is the $n\times
		p$ matrix whose $i$th row is $x_i^T$,
		\item
		$
		\int_{\mathbb{R}_+} u^{d/2} h(u) \, du < \infty
		$, and
		\item
		$
		\int_{\mathbb{R}_+} u^{-(n-p-d)/2} h(u) \, du < \infty
		$.
	\end{enumerate}
	The posterior density is highly intractable, but there is a
	well-known DA algorithm to sample from it
	\citep{liu1996bayesian}. Under our framework, the DA chain
	$\Phi$ is characterized by the Mtd
	\[
	p\big((\beta,\Sigma),(\cdot,\cdot) \big) = \int_{\mathbb{R}_+^n} \pi_{U|V}(\cdot,\cdot| z) \pi_{V|U}(z|\beta,\Sigma) \, dz,
	\]
	where $z=(z_1,z_2,\dots,z_n)^T$,
	\[
	\begin{aligned}
	\pi_{U|V}(\beta,\Sigma|z) & \propto |\Sigma|^{-(n+d+1)/2} \prod_{i=1}^{n}\exp \Big\{ -\frac{z_i}{2} \big(y_i-\beta^Tx_i \big)^T \Sigma^{-1} \big(y_i-\beta^Tx_i \big) \Big\}, \, \mbox{and} \\
	\pi_{V|U}(z|\beta,\Sigma) & \propto \prod_{i=1}^{n}
	z_i^{d/2} \exp \Big\{ -\frac{z_i}{2} \big(y_i-\beta^Tx_i
	\big)^T \Sigma^{-1} \big(y_i-\beta^Tx_i \big) \Big\} h(z_i)
	\,.
	\end{aligned}
	\]
	The first conditional density, $\pi_{U|V}(\cdot,\cdot|z)$,
	characterizes a multivariate normal distribution on top of an
	inverse Wishart distribution, i.e. $\beta|\Sigma,z$ is
	multivariate normal, and $\Sigma|z$ is inverse Wishart.  The
	second conditional density, $\pi_{V|U}(\cdot|\beta,\Sigma)$,
	is a product of $n$ univariate densities.  Moreover, when
	$h(\cdot)$ is a standard pdf on $\mathbb{R}_+$, these
	univariate densities are often members of a standard
	parametric family.  The following proposition about the
	resulting DA operator is proved in \cite{qin2018trace}.
	\begin{proposition} \label{nongaussian} Suppose $h(\cdot)$ is
		strictly positive in a neighborhood of the origin. If there
		exists $\xi \in (1,2)$ and $\delta>0$ such that
		\[
		\int_{0}^{\delta} \frac{u^{d/2}h(u)}{\int_{0}^{\xi u} v^{d/2} h(v) \, dv} \, du < \infty, 
		\]
		then $P$ is trace-class.
	\end{proposition}
	
	When $P$ is trace-class, we can pick an $\omega(\cdot)$ and
	try to make use of~\eqref{estimator}. A sufficient condition
	for the estimator's variance, $D^2$, to be finite is stated in the
	following proposition, whose proof is given in the appendix.
	\begin{proposition} \label{nongaussianmoment} Suppose that
		$h(\cdot)$ is strictly positive in a neighborhood of the
		origin.  If $\omega(z)$ can be written as $\prod_{i=1}^{n}
		\omega_i(z_i),$ and there exists $\xi \in (1,4/3)$ such that
		for all $i \in \{1,2,\dots,n\},$
		\begin{equation} \label{nongaussianomega}
		\int_{\mathbb{R}_+} \frac{u^{3d/2}h^3(u)} {(\int_{0}^{\xi u} v^{d/2}h(v) \, dv)^3 \omega_i^2(u) } \, du < \infty,
		\end{equation}
		then \eqref{momentcond} holds, and thus by Theorem~\ref{secondmoment}, the estimator \eqref{estimator} has finite variance.
	\end{proposition}
	
	For illustration, take $d=1$ and $h(u) \propto
	u^{-2}e^{-1/(8u)}$.  Then $\varepsilon_1$ follows a scaled
	Laplace distribution, and the model can be viewed as a median
	regression model with variance $\Sigma$ unknown.  It's easy to
	show that $h(\cdot)$ satisfies the assumptions in
	Proposition~\ref{nongaussian}, so the resultant DA operator is
	trace-class.  Now let
	\[
	\omega(z) = \prod_{i=1}^{n} \omega_i(z_i) \propto
	\prod_{i=1}^{n} z_i^{-3/2} e^{-1/(32z_i)} \,.
	\]
	The following result shows that this will lead to an estimator
	with finite variance.
	\begin{corollary}
		Suppose $d=1$, $h(u) \propto u^{-2}e^{-1/(8u)}$, and 
		\[
		\omega(z) = \prod_{i=1}^{n} \omega_i(z_i) \propto
		\prod_{i=1}^{n} z_i^{-\alpha-1} e^{-\gamma/z_i} \,,
		\]
		where $0 < \alpha < 3/4$ and $0 < \gamma < 3/64$.  Then
		the variance, $D^2$, is finite.
	\end{corollary}
	\begin{proof}
		In light of Proposition~\ref{nongaussianmoment}, we only
		need to show that \eqref{nongaussianomega} holds for some
		$\xi \in (1,4/3).$ For any $\xi > 0,$ making use of the fact that (by monotone
		convergence theorem)
		\[
		\lim\limits_{u \to \infty} \int_{0}^{\xi u}
		v^{1/2}h(v) \, dv = \int_{\mathbb{R}_+} u^{1/2} h(u)
		\, du > 0 \,,
		\]
		one can easily show for any $\delta > 0$,
		\begin{equation} \label{inftyend}
		\int_{\delta}^{\infty} \frac{u^{3/2}h^3(u)} {(\int_{0}^{\xi u} v^{1/2}h(v) \, dv)^3 \omega_i^2(u) } \, du = \int_{\delta}^{\infty} \frac{u^{2\alpha-5/2}\exp\{ 2\gamma/u-3/(8u) \} } {(\int_{0}^{\xi u} v^{1/2}h(v) \, dv)^3} \, du < \infty.
		\end{equation}
		On the other hand, using L'H\^{o}pital's rule, we can see for $(1-16\gamma/3)^{-1} < \xi < 4/3,$
		\[
		\begin{aligned}
		\lim\limits_{u \to 0} \Bigg( \frac{u^{3/2}h^3(u)} {(\int_{0}^{\xi u} v^{1/2}h(v) \, dv)^3 \omega_i^2(u) } \Bigg)^{1/3} &= \lim\limits_{u \to 0} \frac{u^{2\alpha/3-5/6} \exp\{ 2\gamma/(3u)-1/(8u) \}}{ \int_{0}^{\xi u} v^{-3/2}e^{-1/(8v)}\, dv } \\
		&= \lim\limits_{u \to 0} R(u) \exp \Big\{ -\Big(-\frac{2\gamma}{3} - \frac{1}{8\xi} + \frac{1}{8} \Big)\frac{1}{u} \Big\} \\
		& = 0,
		\end{aligned}
		\]
		where $R(u)$ is a function that is either bounded near
		the origin or goes to $\infty$ at the rate of some
		power function as $u \to 0.$ It follows that for $\xi \in ((1-16\gamma/3)^{-1}, 4/3)$ and
		small enough $\delta,$
		\begin{equation} \label{origend}
		\int_{0}^{\delta} \frac{u^{3/2}h^3(u)} {(\int_{0}^{\xi u} v^{1/2}h(v) \, dv)^3 \omega_i^2(u) } \, du < \infty.
		\end{equation}
		Combining \eqref{inftyend} and \eqref{origend} yields
		\eqref{nongaussianomega}. The result then follows.
	\end{proof}

	\begin{figure}
		\centering
		\begin{subfigure}[t]{0.4\textwidth}
			\includegraphics[width=\textwidth]{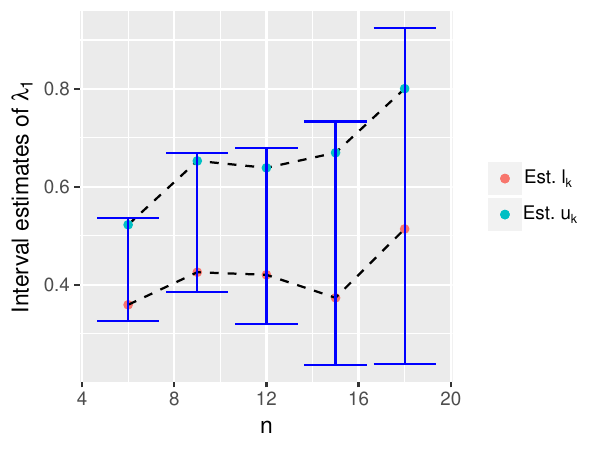}
			\caption{Interval estimates for $\lambda_1$ for different~$n$s.} \label{fig:GrowingN}
		\end{subfigure}
		\hspace{0.5cm}
		\begin{subfigure}[t]{0.4\textwidth}
			\includegraphics[width=\textwidth]{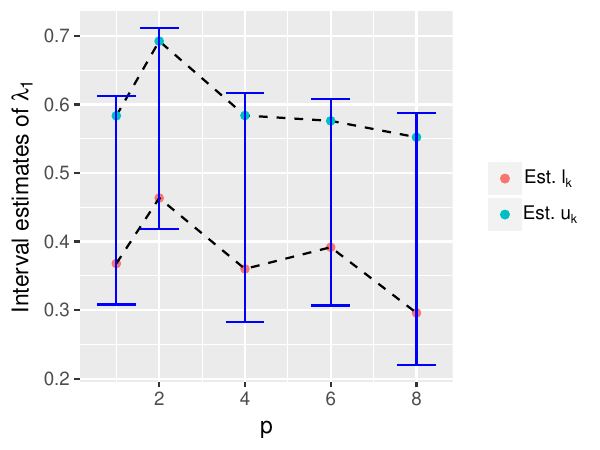}
			\caption{Interval estimates for $\lambda_1$ for different~$p$s.} \label{fig:GrowingP}
		\end{subfigure}
		\caption{Spectral gap estimation for the DA chain for Bayesian linear model}
	\end{figure}
	
	We now test the effectiveness of the Monte Carlo estimator~\eqref{estimator} on a sequence of growing datasets with $d=1$.
	Let $p=3$, and let~$X$ be an $n \times p$ design matrix with~3 distinct rows, $(1,0,0)^T$, $(0,1,0)^T$, and $(0,0,1)^T$, each replicated~$r$ times, so that $n = 3r$.
	The responses, $Y_1,Y_2,\cdots,Y_n$, are then generated according to the previously defined linear regression model with the true value of $\beta$ being $(-3,0,3)^T$, and the true value of~$\Sigma$ being~$1$.
	In other words,
	$
	Y_i - x_i^T (-3,0,3)^T \sim f_h(\cdot)
	$
	independently for each~$i$, where $f_h(u) \propto e^{-|u|/2}$.
	The resultant DA chain~$\Phi$ lives in $S_U = \mathbb{R}^3 \times \mathbb{R}_+$, and $S_V = \mathbb{R}^n = \mathbb{R}^{3r}$.
	Let~$r$ grow from~$2$ to~$6$.
	We use a Monte Carlo sample size of $N = 2 \times 10^6$ to form interval estimates of $\lambda_1$ for different values of~$r$.
	For simplicity, we fix~$k$ to be~$4$.
	The results are given in Figure~\ref{fig:GrowingN}.
	As~$n$ grows, the length of the interval estimate increases quite rapidly.
	This is understandable, since our method is essentially an importance sampling technique, which does not work well in high dimensional settings unless tuned with great care.
	In the previous subsection where we study Bayesian probit regression, we are able to easily deal with a dataset with $n>800$.
	Part of the reason is that, in that case, Algorithm~2 is used, and since $S_U$ is low dimensional, it's easy to choose $\psi(\cdot)$ that resembles $\pi_U(\cdot)$.

	Consider another sequence of datasets where $d=1$, $n=10$, and~$p$ is increased from~$1$ to~$8$.
	The $ij$th element of the design matrix~$X$ is set to be $p_j(i)$, where $\{p_j(\cdot)\}_{j=1}^p$ are orthogonal polynomials generated in \verb|R|.
	The responses are generated according to the aforementioned linear regression model with the true value of $\beta$ being $(-p, p+2p/(1-p), -p + 4p/(1-p), \dots, p)^T$, and the true value of~$\Sigma$ being~$1$.
	In this case, $S_U = \mathbb{R}^p \times \mathbb{R}_+$, and $S_V = \mathbb{R}^{10}$. 
	Using a Monte Carlo sample size of $N=2 \times 10^6$ and setting $k=4$, we obtain interval estimates of $\lambda_1$ for different~$p$s.
	The results are given in Figure~\ref{fig:GrowingP}.
	Compare this to the case where~$p$ is fixed an~$n$ grows.
	We see that the effectiveness of Algorithm~1, characterized by the length of the interval estimate it produces, is much less susceptible to the growing dimension of $S_U$ than to that of $S_V$.

	\vspace*{5mm}
	
	\noindent {\bf \large Acknowledgment}.  The second and third authors were supported by NSF Grant DMS-15-11945.
	
	
	\vspace*{12mm}
	\noindent {\LARGE \bf Appendix}
	\appendix
	
	\section{Proof of Theorem~\ref{basic}}

	\vspace*{3mm}
	\noindent
	{\bf Theorem \ref{basic}.\,}
	\textit{
		The DA operator~$P$ is trace-class if and only if
		\begin{equation} \tag{\ref{traceclass}}
		\int_{S_U} p(u,u) \, \mu(du) < \infty.
		\end{equation}
		If~\eqref{traceclass} holds, then for any positive integer $k,$
		\begin{equation} \tag{\ref{sum=int}}
		s_k := \sum_{i=0}^{\kappa} \lambda_i^k = \int_{S_U} p^{(k)}(u,u) \, \mu(du) < \infty.
		\end{equation}
		}
		
	\begin{proof}
          Note that $P$ is self-adjoint and non-negative.  Let
          $\{g_i\}_{i=0}^{\infty}$ be an orthonormal basis of
          $L^2(\pi_U)$.  The operator $P$ is defined to be trace-class
          if \citep[see e.g.][]{conway2000course}
		\begin{equation} \label{traceclassdef}
		\sum_{i=0}^{\infty} \langle P g_i, g_i \rangle_{\pi_U} < \infty.
		\end{equation}
		This condition is equivalent to $P$ being compact with
                summable eigenvalues.  To show that $P$ being
                trace-class is equivalent to~\eqref{traceclass}, we
                will prove a stronger result, namely
		\begin{equation} \label{trace} \sum_{i=0}^{\infty}
                  \langle P g_i, g_i \rangle_{\pi_U}= \int_{S_U}
                  p(u,u) \mu(du).
		\end{equation}
		
		We begin by defining two new Hilbert spaces.  Let
                $L^2(\pi_V)$ be the Hilbert space consisting of
                functions that are square integrable with respect to
                the weight function $\pi_V(\cdot).$ For $f, g \in
                L^2(\pi_V),$ their inner product is defined, as usual,
                by
		\[
		\langle f, g \rangle_{\pi_V} = \int_{S_V} f(v)
                \overline{g(v)} \pi_V(v) \, \nu(dv).
		\]
		Let $L^2(\pi_U\times\pi_V)$ be the Hilbert space of
                functions on $S_U\times S_V$ that are square
                integrable with respect to the weight function
                $\pi_U(\cdot)\pi_V(\cdot).$ For $f, g \in
                L^2(\pi_U\times \pi_V),$ their inner product is
		\[
		\langle f, g\rangle_{\pi_U\times\pi_V} =
                \int_{S_U\times S_V} f(u,v) \overline{g(u,v)}
                \pi_U(u)\pi_V(v) \,\mu(du)\,\nu(dv).
		\]
		Note that $L^2(\pi_V)$ is separable. Let
                $\{h_j\}_{j=0}^{\infty}$ be an orthonormal basis of
                $L^2(\pi_V).$ It can be shown that $\{g_ih_j\}_{(i,j)
                  \in \mathbb{Z}_+^2}$ is an orthonormal basis of
                $L^2(\pi_U\times\pi_V)$.  Of course, $g_ih_j$ denotes
                the function given by $(g_ih_j)(u,v) = g_i(u)h_j(v).$
		
		The inequality \eqref{traceclass} is equivalent to
		\[
		\int_{S_U \times S_V}
                \bigg(\frac{\pi_{U,V}(u,v)}{\pi_U(u)\pi_V(v)}\bigg)^2
                \pi_U(u) \pi_V(v) \, \mu(du) \, \nu(dv) < \infty,
		\]
		which holds if and only if the function $\varphi:
                S_U\times S_V \to \mathbb{R}$ given by
		\[
		\varphi(u, v) = \frac{\pi_{U,V}(u,v)}{\pi_U(u)\pi_V(v)}
		\]
		is in $L^2(\pi_U \times \pi_V).$
                Suppose~\eqref{traceclass} holds.  Then by Parseval's
                identity,
		\begin{equation*} \nonumber
		\begin{aligned}
                  \int_{S_U} p(u, u) \, \mu(du) &= \langle \varphi, \varphi \rangle_{\pi_U \times \pi_V} \\
                  &= \sum_{(i,j) \in \mathbb{Z}_+^2} |\langle \varphi, g_ih_j \rangle_{\pi_U \times \pi_V}|^2 \\
                  &= \sum_{(i,j)\in \mathbb{Z}_+^2} \Big| \int_{S_U\times S_V} \overline{g_i(u)} \overline{h_j(v)} \pi_{U,V}(u,v) \, \mu(du)\,\nu(dv) \Big|^2 \\
                  &= \sum_{i=0}^{\infty} \sum_{j=0}^{\infty} \Big|
                  \int_{S_V} \Big(\int_{S_U} \overline{g_i(u)}
                  \pi_{U|V}(u|v) \, \mu(du) \Big) \overline{h_j(v)}
                  \pi_V(v) \,\nu(dv) \Big|^2.
		\end{aligned}
		\end{equation*}
		Again by Parseval's identity, this time applied to the function on $S_V$ (and in fact, in $L^2(\pi_V)$ by Jensen's inequality) given by
		\[
		\varphi_i(v) = \int_{S_U} \overline{g_i(u)} \pi_{U|V}(u|v) \, \mu(du),
		\]
		we have
		\begin{equation} \label{trace2}
		\begin{aligned}
                  \int_{S_U} p(u, u) \, \mu(du) &= \sum_{i=0}^{\infty} \sum_{j=0}^{\infty} |\langle \varphi_i, h_j \rangle_{\pi_V}|^2 \\
                  &= \sum_{i=0}^{\infty} \langle \varphi_i, \varphi_i \rangle_{\pi_V} \\
                  &= \sum_{i=0}^{\infty} \int_{S_V} \Big|\int_{S_U} \overline{g_i(u)} \pi_{U|V}(u|v) \, \mu(du) \Big|^2 \pi_V(v) \, \nu(dv) \\
                  &= \sum_{i=0}^{\infty} \int_{S_V} \int_{S_U} \Big(\int_{S_U} g_i(u') \pi_{U|V}(u'|v) \pi_{V|U}(v|u) \, \mu(du') \Big) \overline{g_i(u)} \pi_U(u) \, \mu(du) \, \nu(dv) \\
                  &= \sum_{i=0}^{\infty} \int_{S_U} \Big(\int_{S_U}
                  p(u,u') g_i(u') \, \mu(du') \Big) \overline{g_i(u)}
                  \pi_U(u) \, \mu(du).
		\end{aligned}
		\end{equation}
		Note that the use of Fubini's theorem in the last
                equality can be easily justified by noting that $g_i
                \in L^2(\pi_U)$, and making use of Jensen's
                inequality.  But the right hand side of~\eqref{trace2}
                is precisely $\sum_{i=0}^{\infty} \langle P g_i, g_i
                \rangle_{\pi_U}.$ Hence, \eqref{trace} holds when
                $\int_{S_U} p(u,u) \, \mu(du)$ is finite.
		
		To finish our proof of~\eqref{trace}, we'll
                show~\eqref{traceclassdef}
                implies~\eqref{traceclass}. Assume
                that~\eqref{traceclassdef} holds. Tracing backwards
                along~\eqref{trace2} yields
		\[
		\sum_{(i,j) \in \mathbb{Z}_+^2} |\langle \varphi_i, h_j \rangle_{\pi_V}|^2 < \infty.
		\]
		This implies that the function
		\[
		\tilde{\varphi} := \sum_{(i,j) \in \mathbb{Z}_+^2} \langle \varphi_i, h_j \rangle_{\pi_V} g_i h_j
		\]
		is in $L^2(\pi_{U} \times \pi_V).$ Recall
                that~\eqref{traceclass} is equivalent to $\varphi$
                being in $L^2(\pi_{U} \times \pi_V).$ Hence, it
                suffices to show that $\tilde{\varphi}(u,v) =
                \varphi(u,v)$ almost everywhere. Define a linear
                transformation $T: L^2(\pi_U) \to L^2(\pi_V)$ by
		\[
		Tf(v) = \int_{S_U} f(u) \pi_{U|V}(u|v) \, \mu(du), \quad \forall f \in L^2(\pi_U).
		\]
		By Jensen's inequality, $T$ is bounded, and thus, continuous. For any $g = \sum_{i=0}^{\infty} \alpha_i g_i \in L^2(\pi_{U})$ and $h = \sum_{j=0}^{\infty} \beta_j h_j \in L^2(\pi_{V}),$
		\[
		\begin{aligned}
		&\int_{S_V} \int_{S_U} \varphi(u,v) \overline{g(u)} \overline{h(v)} \, \pi_U(u)\pi_V(v) \, \mu(du) \, \nu(dv) \\
		&= \langle T \overline{g}, h \rangle_{\pi_V} \\
		&= \sum_{i=0}^{\infty} \sum_{j=0}^{\infty} \overline{\alpha_i \beta_j} \langle T \overline{g_i}, h_j \rangle_{\pi_V} \\
		&= \sum_{i=0}^{\infty} \sum_{j=0}^{\infty} \overline{\alpha_i \beta_j} \langle \varphi_i, h_j \rangle_{\pi_V} \\
		&= \langle \tilde{\varphi}, g h \rangle_{\pi_U\times\pi_V} \\
		&= \int_{S_V} \int_{S_U} \tilde{\varphi}(u,v) \overline{g(u)} \overline{h(v)} \, \pi_U(u)\pi_V(v) \, \mu(du) \, \nu(dv),
		\end{aligned}
		\]
		where $\overline{g} \in L^2(\pi_V)$ is given by
                $\overline{g}(u) := \overline{g(u)},$ and
                $\overline{g_i}$ is defined
                similarly for $i \in \mathbb{Z}_+$. This implies that for any $C_1 \in
                \mathcal{U}$ and $C_2 \in \mathcal{V},$
		\[
		\int_{C_1 \times C_2} \varphi(u,v)\, \pi_U(u)\pi_V(v) \, \mu(du) \, \nu(dv) = \int_{C_1 \times C_2} \tilde{\varphi}(u,v)\, \pi_U(u)\pi_V(v) \, \mu(du) \, \nu(dv).
		\]
		Note that
		\begin{equation} \label{tildephi<}
		\int_{S_U \times S_V} |\tilde{\varphi}(u,v)| \, \pi_U(u)\pi_V(v) \, \mu(du) \, \nu(dv) \leq \langle \tilde{\varphi}, \tilde{\varphi} \rangle_{\pi_U \times \pi_V}^{1/2}< \infty.
		\end{equation}
		By~\eqref{tildephi<} and the dominated convergence
                theorem, one can show that
		\[
		\mathcal{A} := \Big\{ C\in\mathcal{U}\times\mathcal{V} \, \Big| \, \int_{C} \varphi(u,v)\, \pi_U(u)\pi_V(v) \, \mu(du) \, \nu(dv) = \int_{C} \tilde{\varphi}(u,v)\, \pi_U(u)\pi_V(v) \, \mu(du) \, \nu(dv) \Big\}
		\]
		is a $\lambda$ system. An application of Dynkin's
                $\pi$-$\lambda$ theorem reveals that
                $\mathcal{\mathcal{U}\times\mathcal{V}} \subset
                \mathcal{A}.$ Therefore, $\tilde{\varphi}(u,v) =
                \varphi(u,v)$ almost everywhere,
                and~\eqref{traceclass} follows.
		
		For the rest of the proof, assume that~$P$ is
                trace-class. This implies that~$P$ is compact, and
                thus admits the spectral decomposition \citep[see
                e.g.][\S 28 Corollary 2.1]{helmberg2014introduction}
                given by
		\begin{equation} \label{decompP}
		P f = \sum_{i=0}^{\kappa} \lambda_i \langle f, f_i \rangle_{\pi_U} f_i, \quad f \in L^2(\pi_U)
		\end{equation}
		where $f_i, \, i=0,1,\dots,\kappa,$ is the normalized eigenfunction corresponding to $\lambda_i.$ By Parseval's identity,
		\[
		\begin{aligned}
                  \sum_{i=0}^{\infty} \langle Pg_i, g_i \rangle_{\pi_U} &= \sum_{i=0}^{\infty} \sum_{j=0}^{\kappa} \lambda_j |\langle g_i, f_j \rangle_{\pi_U}|^2 \\
                  &= \sum_{j=0}^{\kappa} \lambda_j \langle f_j, f_j \rangle_{\pi_U} \\
                  &= \sum_{j=0}^{\kappa} \lambda_j.
		\end{aligned}
		\]
		This equality is in fact a trivial case of Lidskii's
                theorem \citep[see e.g.][]{erdos1974trace,
                  gohberg2012traces}. It follows that~\eqref{sum=int}
                holds for $k=1.$
		
		We now consider the case where $k\geq 2.$ By
                \eqref{decompP} and a simple induction, we have the
                following decomposition for $P^k.$
		\[
		P^k f = \sum_{i=0}^{\kappa} \lambda_i^k \langle f, f_i
                \rangle_{\pi_U} f_i, \quad f \in L^2(\pi_U) \,.
		\]
		Hence $P^k$ is trace-class with ordered positive
                eigenvalues $\{\lambda_i^k\}_{i=0}^{\kappa}.$ Note
                that $P^k$ is a Markov operator whose Mtd is
                $p^{(k)}(u,\cdot), \, u \in S_U.$ Thus, in order to
                show that \eqref{sum=int} holds for $k \geq 2,$ it
                suffices to verify $P^k$ is a DA operator, for then we
                can treat $P^k$ as $P$ and repeat our argument for the
                $k=1$ case. To be specific, we'll show that there
                exists a random variable $\tilde{V}$ taking values on
                $S_{\tilde{V}},$ where $(S_{\tilde{V}},
                \mathcal{\tilde{V}}, \tilde{\nu})$ is a
                $\sigma$-finite measure space and
                $\mathcal{\tilde{V}}$ is countably generated, such
                that for $u \in S_U,$
		\begin{equation} \label{DA} p^{(k)}(u,\cdot) =
                  \int_{S_{\tilde{V}}} \pi_{U|\tilde{V}}(\cdot|v)
                  \pi_{\tilde{V}|U} (v|u) \, \tilde{\nu}(dv),
		\end{equation}
		where $\pi_{\tilde{V}}(\cdot),$ $\pi_{U|\tilde{V}}(\cdot|\cdot),$ and $\pi_{\tilde{V}|U} (\cdot|\cdot)$ have the apparent meanings.
		
		Let $(U_k,V_k)_{k=0}^{\infty}$ be a
		Markov chain. Suppose that $U_0$ has pdf
		$\pi_U(\cdot)$, and for any non-negative integer $k,$
		let $V_k|U_k=u$ have pdf
		$\pi_{V|U}(\cdot|u),$ and let
		$U_{k+1}|V_k=v$ have pdf
		$\pi_{U|V}(\cdot|v).$ It's easy to see
		$\{U_k\}_{k=0}^{\infty}$ is a stationary DA
		chain with Mtd $p(u,\cdot).$ Suppose $k$ is even. The
		pdf of $U_k|U_0 = u$ is
		\[
		p^{(k)}(u, \cdot) = \int_{S_U} p^{(k/2)}(u,u') p^{(k/2)}(u',\cdot) \mu(du).
		\]
		Meanwhile, since the chain is reversible and starts
		from the stationary distribution,
		$U_0|U_{k/2} = u$ has the same
		distribution as $U_{k/2}|U_0=u,$ which
		is just $p^{(k/2)}(u,\cdot).$ Thus,~\eqref{DA} holds
		with $\tilde{V} = U_{k/2}.$ A similar argument
		shows that when~$k$ is odd,~\eqref{DA} holds with
		$\tilde{V} = V_{(k-1)/2}.$
	\end{proof}
	
	\section{Proof of Proposition~\ref{nongaussianmoment}}

	\vspace*{3mm}
	\noindent
	{\bf Proposition \ref{nongaussianmoment}.\;} \textit{ Suppose that
		$h(\cdot)$ is strictly positive in a neighborhood of the
		origin.  If $\omega(z)$ can be written as $\prod_{i=1}^{n}
		\omega_i(z_i),$ and there exists $\xi \in (1,4/3)$ such that
		for all $i \in \{1,2,\dots,n\},$
		\begin{equation*}
                  \int_{\mathbb{R}_+} \frac{u^{3d/2}h^3(u)} {(\int_{0}^{\xi u} v^{d/2}h(v) \, dv)^3 \omega_i^2(u) } \, du < \infty,
		\end{equation*}
		then \eqref{momentcond} holds, and thus by
                Theorem~\ref{secondmoment}, second moment exists for
                the estimator \eqref{estimator}.  }
	
	\begin{proof}
		Let $S_d$ be the set of $d \times d$ positive definite matrices. For any $\beta \in \mathbb{R}^p,$ $\Sigma \in S_d,$ $z\in \mathbb{R}^n,$ and $\xi \in (1,4/3),$
		\[
		\begin{aligned}
		&\pi_{U|V}(\beta,\Sigma|z) \pi^3_{V|U}(z|\beta,\Sigma) \\
		&= \frac{ |\Sigma|^{-(n+d+1)/2} \prod_{i=1}^{n}\exp \{ -z_i (y_i-\beta^Tx_i )^T \Sigma^{-1} (y_i-\beta^Tx_i ) / 2 \} } { \int_{\mathbb{R}^p} \int_{S_d} |\tilde{\Sigma}|^{-(n+d+1)/2} \prod_{i=1}^{n}\exp \{ -z_i (y_i-\tilde{\beta}^Tx_i )^T \tilde{\Sigma}^{-1} (y_i-\tilde{\beta}^Tx_i ) / 2 \} \, d\tilde{\Sigma}\,d\tilde{\beta}  } \times \\
		&\quad\quad \prod_{i=1}^{n} \frac{ z_i^{3d/2} \exp \{ -3z_i (y_i-\beta^Tx_i )^T \Sigma^{-1} (y_i-\beta^Tx_i )/2 \} h^3(z_i) } { \{ \int_0^{\infty} v^{d/2} \exp [ -v (y_i-\beta^Tx_i )^T \Sigma^{-1} (y_i-\beta^Tx_i ) / 2 ] h(v) \, dv \}^3 } \\
		& \leq \frac{ |\Sigma|^{-(n+d+1)/2} \prod_{i=1}^{n}\exp \{ -z_i (y_i-\beta^Tx_i )^T [\Sigma/(4-3\xi)]^{-1} (y_i-\beta^Tx_i ) /2 \} } { \int_{\mathbb{R}^p} \int_{S_d} |\tilde{\Sigma}|^{-(n+d+1)/2} \prod_{i=1}^{n}\exp \{ -z_i (y_i-\tilde{\beta}^Tx_i )^T \tilde{\Sigma}^{-1} (y_i-\tilde{\beta}^Tx_i ) / 2 \} \, d\tilde{\Sigma}\,d\tilde{\beta}  } \times \\
		&\quad\quad \prod_{i=1}^{n} \frac{ z_i^{3d/2} h^3(z_i) } { ( \int_0^{\xi z_i} v^{d/2} h(v) \, dv )^3 }.
		\end{aligned}
		\]
		Note that
		\[
		\begin{aligned}
		&\int_{S_d} |\Sigma|^{-(n+d+1)/2} \prod_{i=1}^{n}\exp \bigg\{ -\frac{z_i}{2} \big(y_i-\beta^Tx_i \big)^T \bigg(\frac{\Sigma}{4-3\xi}\bigg)^{-1} \big(y_i-\beta^Tx_i \big) \bigg\} \, d\Sigma \\
		&= (4-3\xi)^{-nd/2} \int_{S_d} |\Sigma|^{-(n+d+1)/2} \prod_{i=1}^{n}\exp \Big\{ -\frac{z_i}{2} \big(y_i-\beta^Tx_i \big)^T \Sigma^{-1} \big(y_i-\beta^Tx_i \big) \Big\} \, d\Sigma.
		\end{aligned}
		\]
		Thus,
		\[
		\int_{\mathbb{R}^p} \int_{S_d} \pi_{U|V}(\beta,\Sigma|z) \pi^3_{V|U}(z|\beta,\Sigma) \, d\Sigma \, d\beta \leq (4-3\xi)^{-nd/2}  \prod_{i=1}^{n} \frac{ z_i^{3d/2} h^3(z_i) } { ( \int_0^{\xi z_i} v^{d/2} h(v) \, dv )^3 }.
		\]
		The result follows immediately.
	\end{proof}
	
	
	\bibliographystyle{ims}
	\bibliography{qinbib}
\end{document}